\documentclass[11pt]{article}
\usepackage{amscd}
\usepackage{amsmath}
\usepackage{latexsym}
\usepackage{amsfonts}
\usepackage{amssymb}
\usepackage{amsthm}
\usepackage{graphicx}
\usepackage{verbatim}
\usepackage{mathrsfs}
\usepackage{enumerate}
\usepackage{hyperref}
\usepackage{fullpage}
\usepackage{xcolor}

\theoremstyle{plain}\newtheorem{definition}{Definition}[section]
\theoremstyle{plain}\newtheorem{theorem}{Theorem}[section]
\theoremstyle{plain}\newtheorem{lemma}[theorem]{Lemma}
\theoremstyle{plain}\newtheorem{corollary}[theorem]{Corollary}
\theoremstyle{plain}\newtheorem{proposition}[theorem]{Proposition}
\theoremstyle{plain}\newtheorem{remark}{Remark}[section]

\newcommand{\norm}[1]{\left\|#1\right\|}

\newcommand{\R}{\mathbb{R}}
\newcommand{\be}{\begin{equation}}
\newcommand{\ee}{\end{equation}}
 \newcommand{\ba}{\begin{aligned}}
 \newcommand{\ea}{\end{aligned}}

  \newcommand{\ben}{\begin{enumerate}}
   \newcommand{\een}{\end{enumerate}}

\newcommand{\Rmnum}[1]{\expandafter\@slowromancap\romannumeral #1@}
\allowdisplaybreaks

\begin{document}
\title{Remarks on   well-posedness of the generalized surface quasi-geostrophic equation}
\author{Huan Yu\footnote{
  School of Applied Science, Beijing Information Science and Technology University, Beijing, 100192, P.R.China  Email:  yuhuandreamer@163.com},\,\,\,\,Xiaoxin Zheng\footnote{ School of Mathematics and Systems Science, Beihang University, Beijing 100191, P.R. China Email: xiaoxinzheng@buaa.edu.cn}\;~ and\,
	Quansen Jiu\footnote{
	School of Mathematical Sciences, Capital Normal University, Beijing, 100048, P.R. China Email:
	jiuqs@cnu.edu.cn }}
	\date{}
	\maketitle\vspace{-0.85cm}







\begin{abstract}
In this paper, we  are  concerned with  the Cauchy problem of the generalized surface quasi-geostrophic (SQG) equation in which the velocity field is expressed as $u=K\ast\omega$,
 where $\omega=\omega(x,t)$ is an unknown function and $K(x)=\frac{x^\perp}{|x|^{2+2\alpha}}, 0\le\alpha\le \frac12.$ When $\alpha=0$, it is the two-dimensional Euler equations.  When $\alpha=\frac 12$, it corresponds to the inviscid SQG. We will prove that if the existence interval of the smooth solution to the generalized SQG for some $0\le\alpha_0\le\frac12$ is $[0,T]$, then under the same initial data, the existence interval of the generalized SQG with $\alpha$ which is close to $\alpha_0$ will keep  on $[0,T]$. As a byproduct, our results  imply that the construction of the possible singularity  of the smooth solution of the Cauchy problem to the generalized SQG with $\alpha>0$  will be subtle, in comparison with the singularity  presented in  \cite{[KRYZ]}. To prove our main results, the difference between the two solutions and meanwhile the approximation of the singular integrals will be dealt with. Some new uniform estimates with respect to $\alpha$ on the singular integrals and commutator estimates will be shown in this paper.
\end{abstract}
\smallskip
\maketitle


\section{Introduction and Main Results}
We consider the Cauchy problem of the generalized SQG (Surface Quasi-geostrophic) equation in the plane as follows
\begin{equation}\label{SQG}\tag{SQG}
\left\{\ba
&\omega_{t}+u\cdot\nabla \omega =0, \qquad \qquad(x,t)\in \R^{2}\times\R^{+},\\
&u=\nabla^{\perp}(-\Delta)^{-1+\alpha}\omega, \\
&\omega(x,0)=\omega_0 \ea\ \right.
\end{equation}
 with $0\le \alpha\le\frac 12$. Here, according to the second equation in $\eqref{SQG}$, the unknown scalar function   $\omega=\omega(x,t)$ and vector field $u=(u_1(x,t),u_2(x,t))$  can be expressed as the singular integral
\begin{equation}\label{Int}
u(x)=\int_{\R^2} \frac{(x-y)^\perp}{|x-y|^{2+2\alpha}}\omega(y) \,\mathrm{d}y.
\end{equation}
Throughout our paper,  we omit some constants before the singular integral \eqref{Int}  for conciseness. Meanwhile, the expression \eqref{Int} implies that $u$ is divergence-free, that is, $\nabla\cdot u=\partial_{x_1}u_1+\partial_{x_2}u_2=0.$

When $\alpha=0$, it is well-known that  \eqref{SQG} corresponds to the two-dimensional incompressible Euler equations. In this case, the unknown functions  $\omega=\omega(x,t)$ and $u=u(x,t)$ are the vorticity and the velocity field, respectively. When $\alpha=\frac12$, \eqref{SQG} corresponds to the surface quasi-geostrophic (SQG) equation which describes a famous approximation model of the nonhomogeneous fluid flow in a rapidly rotating 3D half-space (see \cite{[P87]}). In this case, the unknown functions  $\omega=\omega(x,t)$ and $u=u(x,t)$  represent  potential temperature and velocity field, respectively. When $0<\alpha<\frac12$, \eqref{SQG} is called the generalized (or modified) SQG equation.

The classical SQG and the generalized SQG equations  have been widely studied in the past years and much more progress has been  made. In \cite{[KYZ],[KRYZ]}, it is proved that the generalized SQG in half space $\R^{2+}=\{x=(x_1,x_2)| x_2>0\}$ has a unique local solution for vortex-patch initial data and will appear  singularity in finite time for some such kind of initial data when $0<\alpha<\frac{1}{24}$. This strongly implies that the SQG equation will appear finite-time singularity (even for smooth initial data) since the velocity has  less regularity when $\alpha=\frac12$. In fact, the singularity or formation of strong fronts has been suggested in \cite{[CMT]} although the rigorous derivations have not been reached so far. We note that the global well-posedness or blow-up of  the SQG equation is an important issue. As pointed out in \cite{[CMT]}, the singularity of the SQG equation will be similar to that of the three-dimensional Euler equations. Concerning the  dissipative SQG equation, which enjoys  a fractional dissipation term $-(-\Delta)^{\beta}\omega$ on the right hand side of the second equation of \eqref{SQG}, the global well-posedness in the critical case $\beta=\frac12$ was proved independently by Caffarelli and Vasseur \cite{[CV]} and by Kiselev, Nazarov and Volberg \cite{[KNV]} (see \cite{[CVi],[KN]} for different approaches). The proof of global regularity for the subcritical case $\beta>\frac12$ is standard (see e.g. \cite{[Con-Wu]}), while in the supercritical case $\beta<\frac12$  the global   regularity of small solutions is obtained (see e.g. \cite{[CMZ],[HK],[Ju],[Wu]}) and the slightly supercritical case is studied recently in \cite{[DKSV]}.

In this paper, our target is to show that, for any $T>0$, if $\{\omega^{\alpha_0}, u^{\alpha_0}\}$ defined on $[0,T]$ is the unique smooth solution  of \eqref{SQG}  for some $0\le\alpha_0\le \frac12$, then there exists $\delta>0$ such that when $0<\alpha<\frac12$ and  $0<|\alpha_0-\alpha|\le \delta$, the problem \eqref{SQG} with same initial data has also a unique smooth solution $\{\omega^\alpha,u^\alpha\}$ defined on $[0,T]$, where we denote the solution of problem \eqref{SQG} corresponding to $0\le \alpha\le \frac12$ by $\{\omega^\alpha, u^\alpha\}$ satisfying $u^\alpha=\nabla^{\perp}(-\Delta)^{-1+\alpha}\omega^\alpha$.  This is motivated by \cite{[C86]} in which it is shown that if the Cauchy problem to the three-dimensional incompressible Euler equations have a unique smooth solution on $[0,T]$, then the corresponding three-dimensional incompressible Navier-Stokes equations with the same initial data will also have a unique smooth solution defined on $[0,T]$ when the viscosity is suitably small. Furthermore, our result  implies that the construction of the possible singularity of the smooth solution of the Cauchy problem to the generalized SQG with $\alpha>0$  will be  subtle (see Corollary \eqref{Cor1+}), in comparison with the singularity result presented in  \cite{[KRYZ]}. To prove our main results, we consider the behavior of the difference between $u^\alpha$ and $u^{\alpha_0}$. Let us denote
 $$\overline{\omega}=\omega^{\alpha}-\omega^{\alpha_{0}}\quad\text{and}\quad\overline{u}=u^{\alpha}-u^{\alpha_{0}},$$
we easily find that the couple $(\overline{\omega},\,\overline{u})$ satisfies
\begin{equation*}
\overline{\omega}_{t}+(u^{\alpha_{0}}\cdot\nabla) \overline{\omega}+(\overline{u}\cdot\nabla) \overline{\omega}+(\overline{u}\cdot\nabla) \omega^{\alpha_{0}} =0.
\end{equation*}
With this equation, we can establish the following $H^s$-estimate of $\overline{\omega}(t)$:
\begin{align*}
\frac12\frac{\mathrm{d}}{\mathrm{d}t}\|\overline{\omega}(t)\|_{H^s}^2=&-\int_{\R^2} J^s(u^{\alpha_{0}}\cdot\nabla \overline{\omega})J^s\overline{\omega}\,\mathrm{d}x-\int_{\R^2} J^s(\overline{u}\cdot\nabla \overline{\omega})J^s\overline{\omega}\,\mathrm{d}x\\
&-\int_{\R^2} J^s(\overline{u}\cdot\nabla \omega^{\alpha_{0}})J^s\overline{\omega}\,\mathrm{d}x.
\end{align*}
The difficulty for us is  how to use the information of $\|\overline{\omega}(t)\|_{H^s}$ to control $\bar u$ in the nonlinear term. This requires us to consider the behavior in different scale in physical space or in different frequency regime in frequency space. Thus, we decompose $\bar u$ in two parts
\begin{equation}\label{diff-1-1}
\begin{split}
\overline{u}=&u^{\alpha}-u^{\alpha_{0}}
\\=&\nabla^{\perp}(-\Delta)^{-1+\alpha}\overline{\omega}
+(\nabla^{\perp}(-\Delta)^{-1+\alpha}-\nabla^{\perp}(-\Delta)^{-1+\alpha_0})\omega^{\alpha_0}
\\:=&\overline{u}_I+\overline{u}_{II},
\end{split}\end{equation}
where $\bar\omega=\omega^\alpha-\omega^{\alpha_0}$. In the above equalities \eqref{diff-1-1}, we see that the part $u_I$ is related to the difference between the solutions, while the part $u_{II}$ corresponds to the difference between the singular integrals. This observation together with the scale analysis enables us to establish some technique propositions, see Proposition \ref{uni-est}, Proposition \ref{add-0} and Proposition \ref{add1} which will play key roles in our proof of Theorem \ref{th2} and Theorem \ref{th3} respectively. More precisely,  in view of Riesz potential (see \eqref{Riesz}), the  term $\overline{u}_I$  can be estimated as
\begin{equation}\label{Riesz1}
\|\overline{u}_I\|_{L^q(\R^2)}\le C(\alpha)\|\bar\omega\|_{L^p(\R^2)}, \quad \frac1q=\frac1p-\frac{1-2\alpha}{2},
\end{equation}
for $0<\alpha<\frac12$. However, when $\alpha\to\frac12$, the constant $C(\alpha)$ in \eqref{Riesz1} will be unbounded. To overcome this difficulty, we establish  Propositions \ref{uni-est}- \ref{add1} to obtain some new uniform estimates as $\alpha\to \frac12$.

Our main results are stated as follows.
\begin{theorem}\label{th1}
Let $0< \alpha_{0}<\frac12.$  Let $\omega^{\alpha_0}$ be a solution of \eqref{SQG} for $0\leq t\leq T$ with $u^{\alpha_0}=\nabla^{\perp}(-\Delta)^{-1+\alpha_0}\omega^{\alpha_0}$ and $\omega_0\in H^{s+1},$ $s>2$. Then, there exists $\delta>0$ depending on $T$ and $\int_0^T \|\omega^{\alpha_0}\|_{H^{s+1}} \,\mathrm{d}t$ such that if $0<\alpha<\frac12$ and  $|\alpha_0-\alpha|\le \delta$,  the solution $\omega^{\alpha}$ to \eqref{SQG} with
$u^{\alpha}=\nabla^{\perp}(-\Delta)^{-1+\alpha}\omega^{\alpha}$ and the same initial data is smooth on $[0,T]$. Moreover, it holds that
\begin{equation*}
\big\|\omega^{\alpha}(t)-\omega^{\alpha_0}(t)\big\|_{H^{s}}\leq C\left(|\alpha_0-\alpha|^{1-2\alpha}+|\alpha_0-\alpha|^{1-2\alpha_0}+|\alpha_0-\alpha||\log|\alpha_0-\alpha||\right),
\end{equation*}
where $C>0$ is a constant depending on $T$ and $\int_0^T\|\omega^{\alpha_0}\|_{H^{s+1}} \,\mathrm{d}t$.
\end{theorem}

\begin{theorem}\label{th1+}
Let $\alpha_{0}=0.$  Let $\omega^{0}$ be a solution of \eqref{SQG} for $0\leq t\leq T$ with $u^{0}=\nabla^{\perp}(-\Delta)^{-1}\omega^{0}$ and $\omega_0\in H^{s+1}\cap L^a$ with $s>2$ and $1\le a<2$. Then, there exists $\delta>0$ depending on $T$ and $\int_0^T \|\omega^{0}\|_{H^{s+1}\cap L^a}\,\mathrm{d}t$ such that if $0<\alpha<\delta$,  the solution $\omega^{\alpha}$ to \eqref{SQG} with
$u^{\alpha}=\nabla^{\perp}(-\Delta)^{-1+\alpha}\omega^{\alpha}$ and the same initial data is smooth on $[0,T]$. Moreover, it holds that
\begin{equation*}
\big\|\omega^{\alpha}(t)-\omega^{0}(t)\big\|_{H^{s}}\leq C\left(\alpha^{1-2\alpha}+\alpha(1+|\log\alpha|)\right),
\end{equation*}
where $C>0$ is a constant depending on $T$ and $\int_0^T\|\omega^{0}\|_{H^{s+1}\cap L^a} dt$.
\end{theorem}

\begin{remark}\label{Rm1+}
The proof of Theorem \ref{th1+} is similar to that of Theorem \ref{th1} but a few estimates must be modified and the initial data is required to satisfy $\omega_0\in H^{s+1}\cap L^a$ with $s>2$ and $1\le a<2$. The additional condition $\omega_0\in L^a$ with  $1\le a<2$ can guarantee that the velocity field $u$ is well-defined and the interpolation inequality is available $($see the proof of Theorem \ref{th1+} in Section 5.2$)$.  Thanks to the incompressible condition $\nabla\cdot u=0$,  the solution will stay in $L^a$ as well.
\end{remark}

The following are concerned with the case $\alpha_0=\frac12$ which corresponds to the SQG equation.
\begin{theorem}\label{th2}
Let $0<\alpha<\alpha_{0}=\frac12.$  Let $\omega^{\alpha_0}$ be a solution of \eqref{SQG} for $0\leq t\leq T$ with $u^{\alpha_0}=\nabla^{\perp}(-\Delta)^{-1+\alpha_0}\omega^{\alpha_0}$ and $\omega_0\in H^{s+1}\cap L^1,$ $s>2$. Then, there exists $\delta>0$ depending on $T$ and $\int_0^T\|\omega^{\alpha_0}\|_{H^{s+1}\cap L^1}\,\mathrm{d}t$ such that if $0<\alpha_0-\alpha<\delta$, the solution $\omega^{\alpha}$ to \eqref{SQG} with
$u^{\alpha}=\nabla^{\perp}(-\Delta)^{-1+\alpha}\omega^{\alpha}$ and the same initial data is smooth on $[0,T]$. Moreover, it holds that
\begin{equation*}
\big\|\omega^{\alpha}(t)-\omega^{\alpha_0}(t)\big\|_{H^{s}}\leq C\left(\Big(\frac12-\alpha\Big)+\Big(\frac12-\alpha\Big)\log^2\Big(\frac12-\alpha\Big)\right),
\end{equation*}
where $C>0$ is a constant depending on $T$ and $\int_0^T\|\omega^{\alpha_0}\|_{H^{s+1}\cap L^1} dt$.
\end{theorem}

\begin{theorem}\label{th3}
	Let $0<\alpha<\alpha_{0}=\frac12.$  Let $\omega^{\alpha_0}$ be a solution of \eqref{SQG} for $0\leq t\leq T$ with $u^{\alpha_0}=\nabla^{\perp}(-\Delta)^{-1+\alpha_0}\omega^{\alpha_0}$ and $\omega_0\in H^{s+2},$ $s>2$. Then, there exists $\delta>0$ depending on $T$ and $\int_0^T \|\omega^{\alpha_0}\|_{H^{s+2}} \,\mathrm{d}t$ such that if $0<\alpha_0-\alpha<\delta$, the solution $\omega^{\alpha}$ to \eqref{SQG} with
	$u^{\alpha}=\nabla^{\perp}(-\Delta)^{-1+\alpha}\omega^{\alpha}$ and the same initial data is smooth on $[0,T]$. Moreover, it holds that
	\begin{equation*}
		\big\|\omega^{\alpha}(t)-\omega^{\alpha_0}(t)\big\|_{H^{s}}\leq C\left(\Big(\frac12-\alpha\Big)+\Big(\frac12-\alpha\Big)\log^2\Big(\frac12-\alpha\Big)\right),
	\end{equation*}
where $C>0$ is a constant depending on $T$ and $\int_0^T \|\omega^{\alpha_0}\|_{H^{s+2}} dt$.
\end{theorem}
\begin{remark}\label{Rm1}
In Theorem \ref{th1} and Theorem \ref{th1+}, we consider the case $0<\alpha_0<\frac12$ and $\alpha_0=0$ respectively.
In Theorem \ref{th2} and Theorem \ref{th3}, we deal with the case $\alpha_0=\frac12$.
 It is noted that in the proof of Theorem \ref{th1} and Theorem \ref{th1+},  the Hardy-Littlewood-Sobolev inequality $($see Lemma \ref{Hardy}$)$ will be used.  One point  is that the expression $\bar u_I$ in \eqref{diff-1-1} can be reduced to $I_{1-2\alpha}\bar\omega$, where $I_{1-2\alpha}$ is a  Riesz operator  $($see  \eqref{Riesz}$)$ which is bounded from $L^p$ to $L^q$ with $\frac1q=\frac1p-\frac{1-2\alpha}{n}$ satisfying $0<1-2\alpha<n$ and $1<p<q<\infty$. However,  it does not hold bounded uniformly with respect to $\alpha$ when $\alpha$ tends to $\frac12$. Hence, in the proof of Theorem \ref{th2} and Theorem \ref{th3} we can not use Hardy-Littlewood-Sobolev inequality directly. To prove Theorem \ref{th2}, we will establish some new and uniform estimates with respect to $\alpha$ on the singular integral \eqref{Int} $($see Proposition \ref{uni-est}$)$. To prove Theorem \ref{th3}, we will obtain some elegant estimates concerning $u_I$  with the help of Besov spaces of which definition is given in Appendix \ref{sec4} $($see Propositions \ref{add-0} and \ref{add1}$)$.
 \end{remark}

\begin{remark}\label{Rm2+}
As mentioned above, \eqref{SQG} becomes the two-dimensional incompressible Euler equations when $\alpha_0=0$, of which the global existence of smooth solutions has been known $($see \cite{Majda1} and references therein$)$. In comparison with the singularity for the patch solution with $0<\alpha<\frac{1}{24}$ in half space  obtained in \cite{[KRYZ]}, whether the patch or smooth solution  of the Cauchy problem  to \eqref{SQG} when $\alpha>0$ appears singularity in finite time remains open. Theorem \eqref{th1} implies that the possible blow-up time of the smooth solution to the Cauchy problem of  \eqref{SQG} with $\alpha>0$ can not be uniformly bounded when $\alpha\to 0+$.  More precisely, as a corollary of Theorem \ref{th1}, we have
\end{remark}

\begin{corollary}\label{Cor1+}
Let $T^*_\alpha>0$ the maximal existence time $($may be $+\infty$$)$ of the  solution $\omega^\alpha\in C([0,T]; H^{s+1}) (s>2)$  to \eqref{SQG} with $\alpha>0$. Then  $\displaystyle\limsup_{\alpha\to 0+}T^*_\alpha=+\infty$.
\end{corollary}

The rest of the paper is organized as follows. In Section \ref{pre}, we will present some basic facts  which will be needed later. In Section \ref{SI}, we will investigate a singular integral which can be viewed as an approximation of the Riesz transform. In Section \ref{BesovE}, we will obtain  nonlinear terms and commutator estimates  related to $u_I=\nabla^{\perp}(-\Delta)^{-1+\alpha}\overline{\omega}$ in \eqref{diff-1-1}. The proof of the main results will given in Section \ref{sec5}. In the end of the paper,  Appendix \ref{sec4} on the Littlewood-Paley decomposition, Besov spaces will be given.

\section{Preliminaries}\label{pre}
\setcounter{section}{2}\setcounter{equation}{0}

In this section, we present some basic analysis facts. First of all, we introduce
\[\Lambda^s=(-\Delta)^\frac{s}{2}\quad\text{and}\quad
J^s=(I-\Delta)^\frac{s}{2},\] where
 $$\widehat{\Lambda^sf}(\xi)=|\xi|^s\widehat{f}(\xi)\quad\text{and}\quad\widehat{J^sf}(\xi)
 =\big(1+|\xi|^2\big)^{\frac{s}{2}}\widehat{f}(\xi),~~ s\in \R.$$

\begin{definition}[\cite{[stein]}]
Let  $s\in \R$ and $1\leq p\leq \infty.$ We write
$$\|f\|_{W^{s,p}(\R^n)}:=\norm{J^s f}_{L^p(\R^n)}, ~~\|f\|_{\dot{W}^{s,p}(\R^n)}:=\norm{\Lambda^s f}_{L^p(\R^n)}.$$
The nonhomogeneous Sobolev space $W^{s,p}(\R^n)$ is defined as
$$
W^{s,p}(\R^n)=\{f\in \mathcal{S'}(\R^n):\|f\|_{W^{s,p}(\R^n)}<\infty\}.
$$
The homogeneous Sobolev space $\dot{W}^{s,p}(\R^n)$ is defined as
$$
\dot{W}^{s,p}(\R^n)=\{f\in \mathcal{S'}(\R^n):\|f\|_{\dot{W}^{s,p}(\R^n)}<\infty\}.
$$
Here $\mathcal{S'}$ is the Schwarz distributional function space.
\end{definition}
With this definition in hand, we give a commutator estimate and product estimate (see, e.g., Kenig, Ponce and Vega \cite{[KPV]}).
\begin{lemma}\label{commutator}
Let $s>0$ and $1<p<\infty.$ Then
\begin{equation}\label{com-1}
\|J^s(fg)\|_{L^p(\R^n)}\leq C(\|f\|_{W^{s,p_1}(\R^n)}\|g\|_{L^{p_2}}+\|g\|_{W^{s,p_3}(\R^n)}\|f\|_{L^{p_4}(\R^n)}),
\end{equation}
\begin{equation}\label{com-2}
\|J^s(fg)-fJ^sg\|_{L^p(\R^n)}\leq C(\|f\|_{W^{s,p_1}(\R^n)}\|g\|_{L^{p_2}(\R^n)}+\|g\|_{W^{s-1,p_3}(\R^n)})\|\nabla f\|_{L^{p_4}(\R^n)},
\end{equation}
where $\frac1p=\frac{1}{p_1}+\frac{1}{p_2}=\frac{1}{p_3}+\frac{1}{p_4}$ with $p_2,$ $p_4\in [1,\infty]$ and $p_1,$ $p_3\in (1,\infty)$, and $C^{'}s$ are constants depending on $s,$
$p_1,$ $p_2,$ $p_3$ and $p_4$. In addition, these inequalities remain valid when $J^s$ is replaced by $\Lambda^s$.
\end{lemma}
We continue with the  Sobolev embedding theorem \cite{[MWZ]}.
\begin{lemma}\label{embedding1}
For $p_2\neq\infty,$ $W^{s_1,p_1}(\R^n)\hookrightarrow W^{s_2,p_2}(\R^n)$ if and only if
$$s_1-\frac{n}{p_1}\geq s_2-\frac{n}{p_2},~~\frac{1}{p_1}\geq\frac{1}{p_2}.$$

\end{lemma}

The following lemma is the so called Hardy-Littlewood-Sobolev inequality of fractional integration \cite{[stein]}. We begin with definition of the Riesz potential $I_{\alpha}$.
\begin{definition}
Let $0<\alpha<n.$ The Riesz potential $I_{\alpha}=(-\Delta)^{-\frac{\alpha}{2}}$ is defined by
\begin{equation}\label{Riesz}
I_{\alpha}f(x)=\frac{1}{\gamma(\alpha)}\int_{\R^n}\frac{f(y)}{|x-y|^{n-\alpha}}\,\mathrm{d}y,
\end{equation}
with
$$\frac{1}{\gamma(\alpha)}=\pi^{\frac n2}2^\alpha \frac{\Gamma(\frac{\alpha}{2})}{\Gamma(\frac n2-\frac \alpha 2)}.$$
\end{definition}

\begin{lemma}\label{Hardy}
Let $0<\alpha<n,$ $1<p<q<\infty$, $\frac1q=\frac1p-\frac{\alpha}{n}.$
Then, there exists a constant $A_{p,q}$ depending on $p,$ $q$ such that
\begin{equation}\label{H}
\|I_{\alpha}f\|_{L^q(\R^n)}\leq A_{p,q}\norm{f}_{L^p(\R^n)}.
\end{equation}

\end{lemma}

\begin{remark}
It is noted that the constant $A_{p,q}$ is unbounded as $\alpha\to 0$ here.
\end{remark}
The following is an elementary result from \cite{[C86]} in which the case $m=1$ is proved.
\begin{proposition}\label{o}
Let $T>0,$ $G>0$ and $m>0$ be given constants and let $F(t)$ be a nonnegative continuous function on $[0,T).$ Let $\nu_0$
be defined by $$\nu_0=\frac{1}{4m(2m TG)^\frac1m\int_0^T F(t)\,\mathrm{d}t}.$$ Then, for all $0<\nu\leq\nu_0,$ all nonnegative solution $y(t)$ of
the system \begin{equation}\label{ODE}
\left\{\ba
&\frac{\mathrm{d}\,y(t)}{\mathrm{d}t}\leq \nu F(t)+Gy(t)^{1+m}\\
&y(0)=0\ea\ \right.
\end{equation} is uniformly bounded on $[0,T)$ and
 \begin{equation}\label{ODE-0}
y(t)\leq\min\bigg\{\frac{4^{\frac1m}-1}{(2m TG)^\frac1m},\,\, 4m\big(4^\frac1m-1\big)\nu\int_0^T F(t)\,\mathrm{d}t\bigg\}.
\end{equation}
\end{proposition}
\begin{proof}
Let us define $$\sigma=\min\bigg\{\frac{1}{(2mT)^{1+\frac1m}G^{\frac1m}}, \Big(4m\nu\int_0^T F(t)\,\mathrm{d}t\Big)^{1+m}G\bigg\}.$$
Dividing the first equation of  \eqref{ODE} by $\big(1+(\frac{G}{\sigma})^{\frac{1}{m+1}}y\big)^{1+m}$  yields
\begin{equation}
\frac1m\Big(\frac{\sigma}{G}\Big)^{\frac{1}{m+1}}\frac{\mathrm{d}}{\mathrm{d}t}\bigg(\frac{1}{\big(1+(\frac{G}{\sigma})^{\frac{1}{m+1}}y\big)^{m}}\bigg)\geq -\nu F(t)-\sigma.
\end{equation}
By integrating from $0$ to $t$, we obtain
\begin{equation}\label{ODE-1}
\frac1m\Big(\frac{\sigma}{G}\Big)^{\frac{1}{m+1}}\bigg(\frac{1}{(1+(\frac{G}{\sigma})^{\frac{1}{m+1}}y)^{m}}\bigg)\geq
 \frac1m\Big(\frac{\sigma}{G}\Big)^{\frac{1}{m+1}}-\nu \int_0^T F(t)\,\mathrm{d}t-\sigma T.
\end{equation}
The choice $\sigma\leq \frac{1}{(2mT)^{1+\frac1m}G^{\frac1m}}$ implies $\sigma T\leq \frac{1}{2m}(\frac{\sigma}{G})^{\frac{1}{m+1}}.$

For $\nu\leq \nu_0,$ we have
$$\nu \int_0^T F(t)\,\mathrm{d}t \leq  \frac{1}{4m}\Big(\frac{\sigma}{G}\Big)^{\frac{1}{m+1}}.$$
Indeed, if $\sigma=\big(4m\nu\int_0^T F(t)\,\mathrm{d}t\big)^{1+m}G,$ the last inequality is indeed an equality and if $\sigma= \frac{1}{(2mT)^{1+\frac1m}G^{\frac1m}},$
it follows from $$\nu \int_0^T F(t)\,\mathrm{d}t \leq\nu_0 \int_0^T F(t)\,\mathrm{d}t=\frac{1}{4m(2m TG)^\frac1m}=\frac{1}{4m}\Big(\frac{\sigma}{G}\Big)^{\frac{1}{m+1}}.$$
Thus, we get by \eqref{ODE-1} that
 \begin{equation*}
\frac{1}{\big(1+(\frac{G}{\sigma})^{\frac{1}{m+1}}y\big)^{m}}\geq\frac14
\end{equation*}
which implies \eqref{ODE-0}.
\end{proof}
\section{Estimates on A Singular Integral}\label{SI}
\setcounter{section}{3}\setcounter{equation}{0}

In this section, we present some new results on a singular integral which will be needed in the proof of Theorem \ref{th2}. We denote
\begin{equation}\label{T-operator}
Tf(x)=K*f(x)=\int_{\R^n} K(x-y)f(y) dy \quad\text{with}\quad K(x)=\frac{x}{|x|^{n+1-\beta}},\quad 0<\beta<n,
\end{equation}
where $x\in \R^n$.

Let \begin{equation}\label{Chi-L}
\chi_\lambda(s)=\chi(\lambda s).
\end{equation}
where $\chi(s)\in C_0^\infty(\R)$ is the usual smooth cutting-off function which is defined as
$$
\chi(s)=
\left\{
\begin{array}{ll}
1, \quad& |s|\le 1,\\[3mm]
0,\quad & |s|\ge 2,
\end{array}
\right.
$$
satisfying $|\chi'(s)|\le 2$.

Setting
\begin{eqnarray}
&& T_1f(x):=K_1*f(x)\quad\text{with}\quad K_1(x)=K(x)\chi_\beta(|x|),\label{T1-operator}\\
&& T_2f(x):=K_2*f(x)\quad\text{with}\quad K_2(x)=K(x)(1-\chi_\beta(|x|)).\label{T2-operator}
\end{eqnarray}
Then we have the following proposition which  holds for general $n$-dimensional case.
\begin{proposition}\label{uni-est}
 There exists a constant $C=C(n,s)$ independent of $\beta$ such that
\begin{equation}\label{T1-1}
\|T_1f\|_{H^s(\R^n)}\leq C\|f\|_{H^s(\R^n)},\quad \, s>0, \,\,0<\beta<n;
\end{equation}
\begin{equation}\label{T2-1}
\|T_2f\|_{L^2(\R^n)}\leq C\frac{\beta^{\frac n2}}{\sqrt{n-2\beta}}\|f\|_{L^1(\R^n)},  \,\,0<\beta<\frac n2;
\end{equation}
\begin{equation}\label{T2-2}
\|T_2f\|_{\dot{H}^s(\R^n)}\leq C(\frac{\beta}{1-\beta}+\frac{2^\beta-1}{\beta}\beta^{-\beta})\|f\|_{\dot{H}^{s-1}(\R^n)}, \quad s\geq1,\,\, 0<\beta<1.
\end{equation}
\end{proposition}
\begin{remark}\label{Rm2-1}
When $n=2$, the result of Proposition \ref{uni-est} holds true if $K(x)$ in \eqref{T-operator} is replaced by $ K(x)=\frac{x^\perp}{|x|^{3-\beta}}$.
\end{remark}

\begin{remark}
It is emphasized that the constants $C$ is independent of $\beta$, and  what is more,
$\frac{\beta^{\frac n2}}{\sqrt{n-2\beta}}$ in \eqref{T2-1} is sufficiently small, $\frac{\beta}{1-\beta}+\frac{2^\beta-1}{\beta}\beta^{-\beta}$ is uniformly bounded
in \eqref{T2-2} when $\beta$ tends to zero.
\end{remark}
\begin{remark}
  When $\beta=0$, it follows from \eqref{T-operator} that $Tf=\mathcal{R} f$ (in the sense that the integral takes principle values), where $\mathcal{R}$ is a Riesz transformation which is a strong $(p,p)$ type operator with $1<p<\infty$, that is,
  \begin{equation}\label{Riesz0}
  \|\mathcal{R} f\|_{L^p}\le C\|f\|_{L^p}, \ 1<p<\infty
   \end{equation}
  for some constant $C>0$. By Proposition \ref{uni-est}, it holds
\begin{equation}\label{Riesz2}
\|Tf\|_{L^2}\le C(\|f\|_{L^2}+\beta\|f\|_{L^1}),
\end{equation}
where $C>0$ is an absolute constant. This means that the estimate \eqref{Riesz2} recovers the corresponding one in \eqref{Riesz0} with $p=2$.
\end{remark}
\begin{proof}[Proof of Proposition \ref{uni-est}]
We firstly prove \eqref{T2-1} and \eqref{T2-2}.
Note that
\begin{equation*}
T_2f(x)=\int_{\R^n}\frac{x-y}{|x-y|^{n+1-\beta}}\big(1-\chi_\beta(|x-y|)\big)f(y)\,\mathrm{d}y.
\end{equation*}
For $0<\beta<\frac n2$, we have
\begin{equation*}
\begin{split}
\big\|T_2f\big\|_{L^2(\R^n)}&\leq\Big\|\int_{|x-y|\geq\frac{1}{\beta}}\frac{1}{|x-y|^{n-\beta}}|f(y)|\,\mathrm{d}y\Big\|_{L^2(\R^n)}
\\&\leq\|f\|_{L^1(\R^n)}\Big(\int_{|z|\geq\frac{1}{\beta}}\frac{1}{|z|^{2(n-\beta)}}\,\mathrm{d}z\Big)^{\frac12}
\\&\leq\frac{C}{\sqrt{n-2\beta}}\beta^{\frac n2-\beta}\|f\|_{L^1(\R^n)}.
\end{split}
\end{equation*}
Since
$$\lim_{\beta\to0+}\beta^{-\beta}=1,$$
there exists an absolute constant $C(n)>0$ such that, for any $0\leq\beta<\frac n2,$
\begin{equation*}\begin{split}
\big\|T_2f\big\|_{L^2(\R^n)}&\leq  C\frac{\beta^{\frac n2}}{\sqrt{n-2\beta}}\big\|f\big\|_{L^1(\R^n)}.
\end{split}
\end{equation*}
This means \eqref{T2-1}.

To prove \eqref{T2-2}, we note that for $s\ge 1$ and $i=1, 2, \cdots, n$,
\begin{equation}\label{T2-E}
 \begin{split}
\partial_i\Lambda^{s-1}T_2f(x)=&\int_{\R^n}\partial_i\Big(\frac{x-y}{|x-y|^{n+1-\beta}}\big(1-\chi_\beta(|x-y|)\big)\Big)\Lambda^{s-1}_yf(y)\,\mathrm{d}y
\\=&\int_{|x-y|\geq\frac{1}{\beta}}\partial_i\Big(\frac{x-y}{|x-y|^{n+1-\beta}}\Big)\big(1-\chi_\beta(|x-y|)\big)\Lambda^{s-1}_yf(y)\,\mathrm{d}y\\&
-\int_{\frac{1}{\beta}\leq|x-y|\leq\frac{2}{\beta}}\frac{x-y}{|x-y|^{n+1-\beta}}\partial_i\chi_\beta(|x-y|)\Lambda^{s-1}_yf(y)\,\mathrm{d}y
\\:=&J_1+J_2,
\end{split}
\end{equation}
where $\partial_i=\partial_{x_i}, i=1,2,\ldots, n$.

Then, for $0<\beta<1$, we obtain
\begin{equation}\label{J1}
 \begin{split}
\big\|J_1\big\|_{L^2(\R^n)}&\leq C\Big\|\int_{|x-y|\geq\frac{1}{\beta}}\frac{1}{|x-y|^{n+1-\beta}}|\Lambda^{s-1}_yf(y)|\,\mathrm{d}y\Big\|_{L^2(\R^n)}
\\&\leq C\big\|\Lambda^{s-1}f\big\|_{L^2(\R^n)}\int_{|z|\geq\frac{1}{\beta}}\frac{1}{|z|^{n+1-\beta}}\,\mathrm{d}z
\\&\leq\frac{C}{1-\beta}\beta^{1-\beta}\big\|\Lambda^{s-1}f\big\|_{L^2(\R^n)}.
\end{split}
\end{equation}
The term $J_2$ can be bounded as
\begin{equation}\label{J2}
 \begin{split}
\big\|J_2\big\|_{L^2(\R^n)}&\leq C\Big\|\int_{\frac{1}{\beta}\leq|x-y|\leq\frac{2}{\beta}}\frac{1}{|x-y|^{n-\beta}}|\Lambda^{s-1}_yf(y)|\,\mathrm{d}y\Big\|_{L^2(\R^n)}
\\&\leq C\big\|\Lambda^{s-1}f\big\|_{L^2(\R^n)}\int_{\frac{1}{\beta}\leq|z|\leq\frac{2}{\beta}}\frac{1}{|z|^{n-\beta}}\,\mathrm{d}z
\\&\leq C\frac{2^\beta-1}{\beta}\beta^{-\beta}\big\|\Lambda^{s-1}f\big\|_{L^2(\R^n)}.
\end{split}
\end{equation}
Substituting \eqref{J1} and \eqref{J2} into \eqref{T2-E} and using the fact that $$\|\Lambda^sT_2f\|_{L^2}\le \|\nabla\Lambda^{s-1}T_2f\|_{L^2},$$ we finish the proof of \eqref{T2-2}.

Now we turn to prove \eqref{T1-1}. To do this, it suffice to show that there exists an absolute constant $C>0$ independent $\beta$ such that
\begin{equation}\label{K1-F}
\big\|\widehat{K_1}(y)\big\|_{L^\infty(\R^n)}\leq C, \, \, 0<\beta<n.
\end{equation}
In fact, since $\int_{\mathbb{S}^1} K_1(x) \,\mathrm{d}s(x)=0$ (here $\mathbb{S}^1$ is the unit sphere surface in $\mathbb{R}^n$) and $K_1(x)$ is supported on $|x|\le \frac 2\beta$, we have
\begin{equation}
\widehat{K_1}(y)=\int_{\R^n} e^{2\pi ix\cdot y}K_1(x)\,\mathrm{d}x=\int_{|x|\leq\frac{2}{\beta}} \big(e^{2\pi ix\cdot y}-1\big)K_1(x)\,\mathrm{d}x
\end{equation}
If $|y|<\frac{\beta}{2}$,  it is direct to estimate
\begin{equation}\label{s-1}
\begin{split}
\big|\widehat{K_1}(y)\big|&\leq C|y|\int_{|x|\leq \frac 2\beta}|x|\frac{1}{|x|^{n-\beta}}\,\mathrm{d}x
 \leq \frac{2^\beta}{\beta+1}\beta^{-\beta}.
\end{split}
\end{equation}
Then there exists an absolute constant $C>0$ such that
\begin{equation}\label{K1-F1}
\big|\widehat{K_1}(y)\big|\leq C,\,\,  0<\beta<n,\ \ |y|< \frac{\beta}{2}.
\end{equation}
For $\frac{\beta}{2}\leq |y|\leq\beta,$ we rewrite $\widehat{K_1}(y)$ as
\begin{equation*}\begin{split}
\widehat{K_1}(y)&=\int_{|x|<\frac{1}{|y|}} e^{2\pi ix\cdot y}K_1(x)\,\mathrm{d}x+\int_{\frac{1}{|y|}\leq|x|\leq\frac{2}{\beta}} e^{2\pi ix\cdot y}K_1(x)\,\mathrm{d}x
\\&=\int_{|x|<\frac{1}{|y|}} \big(e^{2\pi ix\cdot y}-1\big)K_1(x)\,\mathrm{d}x+\int_{\frac{1}{|y|}\leq|x|\leq\frac{2}{\beta}} e^{2\pi ix\cdot y}K_1(x)\,\mathrm{d}x
\end{split}
\end{equation*}
Similar to \eqref{s-1}, it deduces
\begin{equation*}\begin{split}
\Big|\int_{|x|<\frac{1}{|y|}}\big (e^{2\pi ix\cdot y}-1\big)K_1(x)\,\mathrm{d}x\Big|&\leq C|y|\int_{|x|<\frac{1}{|y|}}|x|\frac{1}{|x|^{n-\beta}}\,\mathrm{d}x
\\&\leq \frac{1}{\beta+1}\frac{1}{|y|^\beta} \leq \frac{2^\beta}{\beta+1}\beta^{-\beta}
\end{split}
\end{equation*}
and
\begin{equation*}
\begin{split}
\Big|\int_{\frac{1}{|y|}\leq|x|\leq\frac{2}{\beta}} e^{2\pi ix\cdot y}K_1(x)\,\mathrm{d}x\Big|&\leq \int_{\frac{1}{\beta}\leq|x|\leq\frac{2}{\beta}}\frac{1}{|x|^{n-\beta}}\,\mathrm{d}x
 \leq C\frac{2^\beta-1}{\beta}\beta^{-\beta}.
\end{split}
\end{equation*}
Consequently, there exists an absolute constant $C>0$ such that
\begin{equation}\label{K1-F2}
\big|\widehat{K_1}(y)\big|\leq C\left(\frac{2^\beta}{\beta+1}\beta^{-\beta}+\frac{2^\beta-1}{\beta}\beta^{-\beta}\right)\le C, \,\, 0<\beta<n, \,\,\frac{\beta}{2}\le |y|\le\beta.
\end{equation}
As for $ |y|>\beta,$ $\widehat{K_1}(y)$ can be divided into
\begin{equation}\label{K1-F31}
\begin{split}
\widehat{K_1}(y)&=\int_{|x|<\frac{2}{|y|}} e^{2\pi ix\cdot y}K_1(x)\,\mathrm{d}x+\int_{\frac{2}{|y|}\leq|x|\leq\frac{2}{\beta}} e^{2\pi ix\cdot y}K_1(x)\,\mathrm{d}x
\\&=\int_{|x|<\frac{2}{|y|}} \big(e^{2\pi ix\cdot y}-1\big)K_1(x)\,\mathrm{d}x+\int_{\frac{2}{|y|}\leq|x|\leq\frac{2}{\beta}} e^{2\pi ix\cdot y}K_1(x)\,\mathrm{d}x.
\end{split}
\end{equation}
For the first term on the right hand of the above equality, we easily find that
\begin{equation}\label{K1-F311}
\begin{split}
\Big|\int_{|x|<\frac{2}{|y|}} \big(e^{2\pi ix\cdot y}-1\big)K_1(x)\,\mathrm{d}x\Big|&\leq C|y|\int_{|x|<\frac{2}{|y|}}|x|\frac{1}{|x|^{n-\beta}}\,\mathrm{d}x
\\&\leq \frac{C}{(\beta+1)|y|^\beta}\leq \frac{C}{\beta+1}\beta^{-\beta}.
\end{split}
\end{equation}
For the second term,
we  choose $z=\frac{y}{2|y|^2}$ with $|z|=\frac{1}{2|y|}<\frac{1}{2\beta}$ such that $e^{2\pi iy\cdot z}=-1$
and
\begin{equation*}
\int_{\R^n} e^{2\pi ix\cdot y}K_1(x)\,\mathrm{d}x=\frac12\int_{\R^n} e^{2\pi ix\cdot y}\big(K_1(x)-K_1(x-z)\big)\,\mathrm{d}x,
\end{equation*}
moreover, we have
\begin{equation}\label{K1-F321}
\begin{split}
\int_{\frac{2}{|y|}\leq|x|\leq\frac{2}{\beta}} e^{2\pi ix\cdot y}K_1(x)\,\mathrm{d}x=&\frac12\int_{\frac{2}{|y|}\leq|x|\leq\frac{2}{\beta}} e^{2\pi ix\cdot y}\big(K_1(x)-K_1(x-z)\big)\,\mathrm{d}x
\\&-\frac12\int_{\frac{2}{|y|}\leq|x+z|,~ |x|\leq\frac{2}{|y|}} e^{2\pi ix\cdot y}K_1(x)\,\mathrm{d}x
\\&+\frac12\int_{|x+z|\leq\frac{1}{|y|},~ |x|\geq\frac{2}{|y|}} e^{2\pi ix\cdot y}K_1(x)\,\mathrm{d}x\\
&+\frac12\int_{|x+z|\ge\frac{2}{\beta}} e^{2\pi ix\cdot y}K_1(x) \,\mathrm{d}x
\\:= &I+J+K+L.
\end{split}
\end{equation}
To estimate the term $I$, we see that
\begin{equation}\label{F321-I}
\begin{split}
2I&=\int_{\frac{2}{|y|}\leq|x|<\frac{1}{\beta},~|x-z|\leq\frac{1}{\beta}}\Big(\frac{x}{|x|^{n+1-\beta}}-\frac{x-z}{|x-z|^{n+1-\beta}}\Big)e^{2\pi ix\cdot y}\,\mathrm{d}x
\\&+\int_{\frac{1}{\beta}\leq|x|\leq\frac{2}{\beta},~|x-z|\leq\frac{1}{\beta}}\Big(\frac{x}{|x|^{n+1-\beta}}\chi_\beta(x)-\frac{x-z}{|x-z|^{n+1-\beta}}\Big)e^{2\pi ix\cdot y}\,\mathrm{d}x
\\&+\int_{\frac{2}{|y|}\leq|x|<\frac{1}{\beta},~|x-z|\geq\frac{1}{\beta}}\Big(\frac{x}{|x|^{n+1-\beta}}-\frac{x-z}{|x-z|^{n+1-\beta}}\chi_\beta(x-z)\Big)e^{2\pi ix\cdot y}\,\mathrm{d}x
\\&+\int_{\frac{1}{\beta}\leq|x|\leq\frac{2}{\beta},~|x-z|\geq\frac{1}{\beta}}\Big(\frac{x}{|x|^{n+1-\beta}}\chi_\beta(x)-\frac{x-z}{|x-z|^{n+1-\beta}}\chi_\beta(x-z)\Big)e^{2\pi ix\cdot y}\,\mathrm{d}x
\\&:=I_1+I_2+I_3+I_4.
\end{split}
\end{equation}
We first estimate $I_2.$ Thanks to $|x-z|\geq|x|-|z|\geq\frac{1}{\beta}-\frac{1}{2|y|}\geq\frac{1}{2\beta},$ one has
\begin{equation}\label{I-2}
\begin{split}
|I_2|&\leq\int_{\frac{1}{\beta}\leq|x|\leq\frac{2}{\beta}}\frac{1}{|x|^{n-\beta}}\,\mathrm{d}x
+\int_{\frac{1}{2\beta}\leq|x-z|\leq\frac{1}{\beta}}\frac{1}{|x-z|^{n-\beta}}\,\mathrm{d}x
\\&\leq C\frac{2^\beta-1}{\beta}\beta^{-\beta}+C\frac{1-2^{-\beta}}{\beta}\beta^{-\beta}.
\end{split}
\end{equation}
Then, thanks to  $|x|=|x-z+z|\geq|x-z|-|z|\geq\frac{1}{\beta}-\frac{1}{2|y|}\geq\frac{1}{2\beta},$ $I_3$ can be estimated as follows:
\begin{equation}\label{I-3}
\begin{split}
|I_3|&\leq\int_{\frac{1}{2\beta}\leq|x|\leq\frac{1}{\beta}}\frac{1}{|x|^{n-\beta}}\,\mathrm{d}x
+\int_{\frac{1}{\beta}\leq|x-z|\leq\frac{2}{\beta}}\frac{1}{|x-z|^{n-\beta}}\,\mathrm{d}x
\\&\leq C\frac{1 -2^{-\beta}}{\beta}\beta^{-\beta}+C\frac{2^\beta-1}{\beta}\beta^{-\beta}.
\end{split}
\end{equation}
The term $I_4$ is directly estimated as
\begin{equation}\label{I-4}
\begin{split}
|I_4|&\leq\int_{\frac{1}{\beta}\leq|x|\leq\frac{2}{\beta}}\frac{1}{|x|^{n-\beta}}\,\mathrm{d}x
+\int_{\frac{1}{\beta}\leq|x-z|\leq\frac{2}{\beta}}\frac{1}{|x-z|^{n-\beta}}\,\mathrm{d}x
\\&\leq C\frac{2^\beta-1}{\beta}\beta^{-\beta}.
\end{split}
\end{equation}
Now we deal with $I_1.$  First of all, let
$$
f(x)=\frac{x}{|x|^{n+1-\beta}},\qquad f_i(x)=\frac{x_i}{|x|^{n+1-\beta}},\quad i=1,2,\ldots,n.
$$
Then
\begin{equation*}
\partial_jf_i(x)=\frac{\mathbf{\delta}_{ij}}{|x|^{n+1-\beta}}
+(-n-1+\beta)\frac{x_jx_i}{|x|^{n+3-\beta}},\quad i,j=1,2,\ldots,n,
\end{equation*}
where $\delta_{ij}=1$ if $i=j$ and  $\delta_{ij}=0$ if $i\neq j$.
In this case, since $|x-z|\geq|x|-|z|\geq4|z|-|z|\geq3|z|,$ it concludes that for any $0\le t\le 1$, we have $|x-tz|\ge |x-z|-(1-t)|z|\ge \frac12|x-z|$. By the mean value theorem, one has
\begin{equation*}\begin{split}
 |f_i(x-z)-f_i(x)|&=\bigg|\int_0^1 \nabla f_i(x-tz)\cdot z\, \mathrm{d}t\bigg|\\
 &\le C\bigg|\int_0^1 \frac{|z|}{|x-tz|^{n+1-\beta}}\, \mathrm{d}t\bigg|\\
 &\le C\frac{|z|}{|x-z|^{n+1-\beta}}
\end{split}
\end{equation*}
for $i=1,2,\ldots,n,$ and
\begin{equation*}\begin{split}
 |f(x-z)-f(x)|&\le \sum_{j=1}^n |f_j(x-z)-f_j(x)|\le  C\frac{|z|}{|x-z|^{n+1-\beta}}
\end{split}
\end{equation*}
for some absolute constant $C$ depending on $n$.
Consequently,
\begin{equation}\label{I-1}
\begin{split}
|I_1|\leq &C|z|\int_{3|z|\leq|x-z|<\frac{1}{\beta}<\infty}\frac{1}{|x-z|^{n+1-\beta}}\,\mathrm{d}x\\
\\ \leq & C\frac{|z|^\beta}{1-\beta}\leq C\frac{\beta^{-\beta}}{1-\beta}.
\end{split}
\end{equation}
Substituting \eqref{I-2}-\eqref{I-1} into \eqref{F321-I} yields
\begin{equation}\label{I-E}
\begin{split}
|I|\le  C\Big(\frac{2^\beta-1}{\beta}\beta^{-\beta}+\frac{1-2^{-\beta}}{\beta}\beta^{-\beta}+\frac{\beta^{-\beta}}{1-\beta}\Big)
\end{split}
\end{equation}
for some absolute constant $C>0$.

Concerning the term $J$, thanks to $|x|\geq|x+z|-|z|\geq 4|z|-|z|\geq 3|z|$, one has
\begin{equation}\label{J-E}
|J|\leq\int_{3|z|\leq|x|\leq4|z|}\frac{1}{|x|^{n-\beta}}\,\mathrm{d}x
\leq C\frac{1-(\frac34)^{-\beta}}{\beta}\beta^{-\beta}.
\end{equation}
Utilizing $|x|\leq|x+z|+|z|\leq 4|z|+|z|\leq 5|z|$, the term $K$ can be bounded by
\begin{equation}\label{K-E}
|K|\leq\int_{4|z|\leq|x|\leq5|z|}\frac{1}{|x|^{n-\beta}}\,\mathrm{d}x
\leq C\frac{1-(\frac45)^{-\beta}}{\beta}\beta^{-\beta}.
\end{equation}
As for the term $L$, the fact that $\frac{2}{\beta}\ge |x|\ge |x+z|-|z|\ge \frac{2}{\beta}-\frac{1}{2\beta}=\frac{3}{2\beta}$ enables us to conclude
\begin{equation}\label{L-E}
\begin{split}
|L| \leq\int_{\frac{3}{2\beta}\leq|x|\leq\frac{2}{\beta}}\frac{1}{|x|^{n-\beta}}\,\mathrm{d}x
 \leq\frac{1}{\beta}\left(\Big(\frac{3}{2\beta}\Big)^\beta-\Big(\frac{2}{\beta}\Big)^\beta\right).
\end{split}
\end{equation}
Substituting \eqref{I-E}-\eqref{L-E} into \eqref{K1-F321}, we readily obtain that there exists an absolute constant $C>0$ such that
\begin{equation}\label{K12}
\Big|\int_{\frac{2}{|y|}\leq|x|\leq\frac{2}{\beta}} e^{2\pi ix\cdot y}K_1(x)\,\mathrm{d}x\Big|\le C.
\end{equation}
In view of \eqref{K1-F311}, \eqref{K12} and \eqref{K1-F31}, there exists an absolute constant $C>0$ such that
\begin{equation}\label{K1-F3}
\big|\widehat{K_1}(y)\big|\le C, \,\ \quad 0<\beta<n,\,\,\quad |y|>\beta.
\end{equation}
Combining  \eqref{K1-F1}, \eqref{K1-F2} with \eqref{K1-F3}, we finish the proof of \eqref{K1-F}. Applying \eqref{K1-F}, one has
\begin{equation*}\label{T1-1+}\begin{split}
&\|T_1f\|_{L^2(\R^n)}=\big\|\widehat{K_1}\widehat{f}\big\|_{L^2(\R^n)}\le C\big\|\widehat{f}\big\|_{L^2(\R^n)}=C\|f\|_{L^2(\R^n)},\\
&\|\Lambda^sT_1f\|_{L^2(\R^n)}=\big\|\widehat{K_1}\widehat{\Lambda^sf}\big\|_{L^2(\R^n)}\le C\big\|\widehat{\Lambda^sf}\big\|_{L^2(\R^n)}=C\|\Lambda^sf\|_{L^2(\R^n)}
\end{split}
\end{equation*}
for any $0<s<n$. Hence \eqref{T1-1} is proved and the proof of the lemma is complete.
\end{proof}
\section{Estimates via Besov Spaces}\label{BesovE}
\setcounter{section}{4}\setcounter{equation}{0}
In  this section, we will establish two key estimates concerning $$\overline{u}_I=\nabla^{\perp}(-\Delta)^{-1+\alpha}\overline{\omega},\qquad 0<\alpha<\frac12$$ in nonhomogeneous Besov spaces (see the Appendix in the end of the paper) which will be needed in the proof of Theorem \ref{th3}.
The first proposition is to deal with the product of two functions. The second  proposition is about a commutator estimate.
\begin{proposition}\label{add-0}
For any  $s>0,$  there exists a constant $C$ depending only on $s$ and $\alpha$ such that
	\begin{equation*}
	\big\|\overline{u}_I\cdot \nabla \omega^{\alpha_0}\big\|_{H^s(\R^2)}\leq C\left(\|\overline{\omega}\|_{L^2(\R^2)}\|\omega^{\alpha_0}\|_{H^{s+2\alpha+1}(\R^2)}
+\|\overline{\omega}\|_{H^s(\R^2)}\|\omega^{\alpha_0} \|_{B^{1+2\alpha}_{2,1}(\R^2)}\right).
	\end{equation*}
\end{proposition}
\begin{remark}
Let us point out that the positive constant $C$ is uniformly bounded as parameter $\alpha$ goes to $\frac12.$
\end{remark}
\begin{proof}[Proof of Proposition \ref{add-0}]
In view of the Bony decomposition, one write
\begin{equation*}
	\overline{u}_I\cdot \nabla \omega^{\alpha_0}=\sum_{i=1}^2\left(T_{\partial_{i} \omega^{\alpha_0}}\overline{u}^i_I +T_{\overline{u}^i_I}\partial_i \omega^{\alpha_0}+R\big(\overline{u}^i_I, \partial_i \omega^{\alpha_0}\big)\right),
\end{equation*} where
$$T_{\partial_i \omega^{\alpha_0}}\overline{u}^i_I=\displaystyle{\sum_{q>0}} \Delta_q\overline{u}^i_IS_{q-1}\partial_i \omega^{\alpha_0}, \quad
T_{\overline{u}^i_I}\partial_i \omega^{\alpha_0}=\displaystyle{\sum_{q>0}}S_{q-1} \overline{u}^i_I\Delta_q\partial_i \omega^{\alpha_0},$$
$$R\big(\overline{u}^i_I, \partial_i \omega^{\alpha_0}\big)=\displaystyle{\sum_{q\geq-1}}\Delta_{q}\overline{u}^i_I \widetilde{\Delta}_{q}\partial_i \omega^{\alpha_0}.$$
According to the H\"{o}lder inequality and  Lemma \ref{B}, we obtain that for $q>0,$
\begin{equation*}
\begin{split}
2^{qs}\big\|\Delta_q\overline{u}^i_I S_{q-1}\partial_i \omega^{\alpha_0}\big\|_{L^2}&\leq 2^{qs}\big\|S_{q-1}\nabla \omega^{\alpha_0} \big\|_{L^\infty}\big\|\Delta_q\overline{u}_I\big\|_{L^2}
\\&\leq 2^{qs}\sum_{-1\leq k\leq q-2}\|\Delta_{k}\nabla \omega^{\alpha_0} \|_{L^\infty}2^{q(-1+2\alpha)}\|\Delta_q\overline{\omega}\|_{L^2}
\\&\leq2^{qs}\|\Delta_q\overline{\omega}\|_{L^2}\sum_{-1\leq k\leq q-2}2^{(q-k)(2\alpha-1)}2^{k(1+2\alpha)}\|\Delta_{k} \omega^{\alpha_0} \|_{L^2}.
\end{split}
\end{equation*}
Therefore, Lemma \ref{annulus} and the Young inequality for series yields
\begin{equation}\label{p-1}
\begin{split}
\norm{T_{\nabla \omega^{\alpha_0}}\overline{u}_I}_{H^s}&\leq C_s\norm{\big\{2^{qs}\big\|S_{q-1}\partial_i \omega^{\alpha_0} \Delta_q\overline{u}^i_I\big\|_{L^2}\big\}_{q>0}}_{\ell^2}
\\&\leq C_s2^{2(2\alpha-1)}\|\overline{\omega}\|_{H^s}\|\omega^{\alpha_0} \|_{B^{1+2\alpha}_{2,1}}.
\end{split}
\end{equation}
Similarly, for $0<\epsilon<2\alpha,$
\begin{align*}
2^{qs}\big\|S_{q-1} \overline{u}^i_I\Delta_q\partial_i \omega^{\alpha_0}\big\|_{L^2}&\leq2^{qs}\|S_{q-1}\overline{u}_I \|_{L^\infty}\|\Delta_q\nabla \omega^{\alpha_0}\|_{L^2}
\\&\leq\sum_{-1\leq k\leq q-2}\|\Delta_{k} \overline{u}_I \|_{L^\infty}\|\Delta_q \omega^{\alpha_0}\|_{L^2}2^{q(s+1)}
\\&\leq C\|\Lambda^{-(1-2\alpha+\epsilon)}\overline{\omega}\|_{L^{\frac{2}{2\alpha-\epsilon}}}\sum_{k\leq q-2}2^{2k\alpha}2^{q(s+1)}\|\Delta_{q} \omega^{\alpha_0} \|_{L^2}
\\&\leq C\|\overline{\omega}\|_{L^2}2^{q(s+1+2\alpha)}\|\Delta_{q} \omega^{\alpha_0} \|_{L^2}\sum_{k\leq q-2} 2^{2\alpha(k-q)},
\end{align*}
where Lemma \ref{Hardy} has been used in the last inequality. In addition, when $\alpha\rightarrow\frac12,$ the constant $C$ is uniformly bounded.
Hence, by Lemma \ref{annulus}, we get
\begin{equation}\label{p-2}
\begin{split}
\big\|T_{\overline{u}^i_I}\partial_i \omega^{\alpha_0}\big\|_{H^s}&\leq C_s\norm{\big\{2^{qs}\|S_{q-1}\overline{u}_I \Delta_q\nabla \omega^{\alpha_0}\|_{L^2}\big\}_{q>0}}_{\ell^2}
\\&\leq C\|\overline{\omega}\|_{L^2} \|\omega^{\alpha_0} \|_{H^{s+1+2\alpha}}.
\end{split}
\end{equation}

For the reminder term, we see that
\begin{equation*}
\begin{split}
2^{qs}\big\|\Delta_{q}R(\overline{u}^i_I, \partial_i \omega^{\alpha_0})\big\|_{L^2}&\leq
\sum_{q\leq q'+N_0}2^{qs}\big\|\Delta_{q}(\Delta_{q'}\overline{u}^i_I \widetilde{\Delta}_{q'}\partial_i \omega^{\alpha_0})\big\|_{L^2}
\\&\leq\sum_{q\leq q'+N_0}2^{qs}\big\|\Delta_{q'}\overline{u}_I\big\|_{L^\infty} \big\|\widetilde{\Delta}_{q'}\nabla \omega^{\alpha_0}\big\|_{L^2}
\\&\leq C\big\|\Lambda^{-(1-2\alpha+\epsilon)}\overline{\omega}\big\|_{L^{\frac{2}{2\alpha-\epsilon}}}\sum_{q\leq q'+N_0}2^{qs}
2^{(1+2\alpha)q'}\big\|\widetilde{\Delta}_{q'}\omega^{\alpha_0}\big\|_{L^2}
\\&\leq C\|\overline{\omega}\|_{L^{2}}\sum_{q\leq q'+N_0}2^{(q-q')s}
2^{(s+1+2\alpha)q'}\|\widetilde{\Delta}_{q'}\omega^{\alpha_0}\|_{L^2}.
\end{split}
\end{equation*}
This ensures that by Lemma \ref{B}, for $s>0,$
\begin{equation}\label{p-3}
\begin{split}
\big\|R(\overline{u}^i_I, \partial_i \omega^{\alpha_0})\big\|_{H^s}&\leq\big\|\big\{2^{qs}\|\Delta_{q}R(\overline{u}^i_I, \nabla \omega^{\alpha_0})\|_{L^2}\big\}_{q\geq-1}\big\|_{\ell^2}
\\&\leq C\|\overline{\omega}\|_{L^2} \big\|\omega^{\alpha_0}\big \|_{H^{s+1+2\alpha}}.
\end{split}
\end{equation}
Collecting   \eqref{p-1}, \eqref{p-2} and \eqref{p-3} above gives the proof of this proposition.
\end{proof}

\begin{proposition}\label{add1}
For any $s>0,$  there exists a  constant $C$ depending only on $s$ such that,
	\begin{equation}\begin{split}\label{add-2}
	&\|J^s(\overline{u}_I\cdot \nabla \overline{\omega})-\overline{u}_I\cdot J^s \nabla \overline{\omega}\|_{L^2(\R^2)}\\ \leq& C\big(\|\overline{\omega}\|_{H^{2\alpha}(\R^2)}^2
+\|\overline{\omega}\|_{H^s(\R^2)}\|\overline{\omega}\|_{H^{2\alpha+1}(\R^2)}+
\|\overline{\omega}\|_{H^s(\R^2)}\|\overline{\omega} \|_{B^{1+2\alpha}_{2,1}(\R^2)}\big).
	\end{split}\end{equation}
In particular, if $s>2,$ then we have
\begin{equation}\label{c-add}
\|J^s(\overline{u}_I\cdot \nabla \overline{\omega})-\overline{u}_I\cdot J^s \nabla \overline{\omega}\|_{L^2(\R^2)}\leq C\|\overline{\omega}\|_{H^{s}(\R^2)}^2.
\end{equation}
\end{proposition}
\begin{proof}
With the help of Bony's decomposition, one writes
\begin{equation*}\begin{split}
&J^s(\overline{u}_I\cdot \nabla \overline{\omega})-\overline{u}_I\cdot J^s \nabla \overline{\omega}\\=&
\sum_{i=1}^2\Big([J^s, T_{\overline{u}^i_I}\partial_i]\overline{\omega}+J^s(T_{\partial_i \overline{\omega}}\overline{u}^i_I)-T_{J^s\partial_i \overline{\omega}}\overline{u}^i_I+J^s\big(R(\overline{u}^i_I,\partial_i \overline{\omega})\big)
-R(\overline{u}^i_I,J^s\partial_i \overline{\omega})\Big).
\end{split}
\end{equation*}
The last two terms can be further decomposed into three parts
\begin{equation*}\begin{split}
&J^s\big(R(\overline{u}^i_I,\partial_i \overline{\omega})\big)-R(\overline{u}^i_I,J^s\partial_i \overline{\omega})
\\=&\sum_{q'\geq0}J^s(\Delta_{q'}\overline{u}^i_I \widetilde{\Delta}_{q'}\partial_i \overline{\omega})-\sum_{q'\geq0}\Delta_{q'}\overline{u}^i_I \widetilde{\Delta}_{q'}J^s\partial_i \overline{\omega}+[J^s, {\Delta}_{-1}\overline{u}^i_I \partial_i]\widetilde{\Delta}_{-1}\overline{\omega}.
\end{split}
\end{equation*}
Next, we are going to establish the standard inner $L^2$-norm of the six terms above one by one.

\textbf{Bounds for the term $[J^s, T_{\overline{u}^i_I}\partial_i]\overline{\omega}$}.
By virtue of Proposition \ref{Dy}, we can rewrite  $[J^s, T_{\overline{u}^i_I}\cdot \partial_i]$ as a convolution operator. Indeed,
\begin{equation*}\begin{split}
 [J^s, T_{\overline{u}^i_I}\partial_i]\overline{\omega}
 =&\sum_{q>0}[J^s\widetilde{\Delta}_q, S_{q-1}{u}^i_I \partial_i]\Delta_q\overline{\omega}
\\=&\sum_{q>0}\int_{\R^2} 2^{2q}G_s(2^qy)\big(S_{q-1}\overline{u}^i_I(x-y)-S_{q-1}\overline{u}^i_I(x)\big)\Delta_q\partial_i\overline{\omega}(x-y)\,\mathrm{d}y
,\end{split}
\end{equation*}
where $G_s$ is the inverse Fourier transform of $\xi\mapsto \langle2^{q}\xi\rangle^{s}\varphi(\xi)$.

From the first order Taylor formula, we deduce that
\begin{equation*}\begin{split}
\big|[J^s, T_{\overline{u}^i_I} \partial_i]\overline{\omega}\big|&\leq
\sum_{q>0}\int_{\R^2}\int_{0}^{1} 2^{2q}|G_s(2^qy)y|
|\nabla S_{q-1}\overline{u}^i_I(x-\tau y)|\big|\Delta_q\partial_i\overline{\omega}(x-y)\big|\,\mathrm{d}\tau\mathrm{d}y.
\end{split}
\end{equation*}
Now, taking the $L^2$ norm of the above inequality, using the fact that $L^2\sim B^0_{2,2}$, and using Lemma \ref{annulus}, we get
\begin{equation*}\begin{split}
&\big\|[J^s, T_{\overline{u}^i_I}\partial_i]\overline{\omega}\big\|_{L^2}
\\ \leq&\Big(\sum_{q>0}\Big\|\int_{\R^2}\int_{0}^{1}
2^{2q}|G_s(2^qy)y|
|\nabla S_{q-1}\overline{u}_I(\cdot-\tau y)|
|\Delta_q\nabla\overline{\omega}(\cdot-y)|\,\mathrm{d}\tau\mathrm{d}y\Big\|_{L^2}^2\Big)^{\frac12}.
\end{split}
\end{equation*}
Adopting to  the fact that the
norm of an integral is less than the integral of the norm and using H\"{o}lder's
inequality yield
\begin{equation*}\begin{split}
&\Big\|\int_{\R^2}\int_{0}^{1}
2^{2q}|G_s(2^qy)y|
\big|\nabla S_{q-1}\overline{u}_I(x-\tau y)\big|
\big|\Delta_q\nabla\overline{\omega}(x-y)\big|\,\mathrm{d}\tau\mathrm{d}y\Big\|_{L^2}\\ \leq&
\int_{\R^2}\int_{0}^{1}
2^{2q}|G_s(2^qy)y|
\|\nabla S_{q-1}\overline{u}_I(\cdot-\tau y)\|_{L^\infty}
\|\Delta_q\nabla\overline{\omega}(\cdot-y)\|_{L^2}\,\mathrm{d}\tau\mathrm{d}y\\ \leq&
2^{qs}\|\nabla S_{q-1}\overline{u}_I\|_{L^\infty}\|\Delta_q\overline{\omega}\|_{L^2},
\end{split}
\end{equation*}where the translation invariance of the Lebesgue measure is used in the last inequality.

Hence, the H\"{o}lder inequality and Bernstein's inequality enable us to conclude that
\begin{equation*}\begin{split}
\big\|[J^s, T_{\overline{u}^i_I}\partial_i]\overline{\omega}\big\|_{L^2}&\leq
\Big(\sum_{q>0}2^{2qs}\|\nabla S_{q-1}\overline{u}_I\|_{L^\infty}^2\|\Delta_q\overline{\omega}\|_{L^2}^2\Big)^{\frac12}
\\&\leq \sup_{q>0}\|\nabla S_{q-1}\overline{u}_I\|_{L^\infty}\|\overline{\omega}\|_{H^s}
\\&\leq C\|\overline{\omega}\|_{B^{1+2\alpha}_{2,1}}\|\overline{\omega}\|_{H^s}.
\end{split}
\end{equation*}

\textbf{Bounds for $J^s(T_{\partial_i \overline{\omega}}\overline{u}^i_I).$} By virtue of the H\"{o}lder inequality and  Bernstein's inequality, we get
\begin{equation*}
\begin{split}
2^{qs}\big\|\Delta_q\overline{u}_I\cdot S_{q-1}\partial_i \overline{\omega}\big\|_{L^2}&\leq 2^{qs}\|S_{q-1}\partial_i \overline{\omega} \|_{L^\infty}\|\Delta_q\overline{u}^i_I\|_{L^2}
\\&\leq 2^{qs}\sum_{k\leq q-2}\|\Delta_{k}\nabla \overline{\omega} \|_{L^\infty}2^{q(-1+2\alpha)}\|\Delta_q\overline{\omega}\|_{L^2}
\\&\leq2^{qs}\|\Delta_q\overline{\omega}\|_{L^2}\sum_{k\leq q-2}2^{(q-k)(2\alpha-1)}2^{k(1+2\alpha)}\|\Delta_{k} \overline{\omega}\|_{L^2}.
\end{split}
\end{equation*}
Hence, we have by Lemma \ref{annulus} that
\begin{equation*}\begin{split}
\big\|J^s(T_{\partial_i \overline{\omega}}\overline{u}^i_I)\big\|_{L^2}&\leq
C_s\norm{\big\{2^{qs}\|\Delta_q\overline{u}^i_I S_{q-1}\partial_i \overline{\omega}\|_{L^2}\big\}_{q>0}}_{\ell^2}
\\&\leq C_s2^{2(2\alpha-1)}\|\overline{\omega}\|_{H^s}\|\overline{\omega}\|_{B^{1+2\alpha}_{2,1}}.
\end{split}
\end{equation*}

\textbf{A similar bound holds for both terms $T_{J^s\partial_i \overline{\omega}}\overline{u}^i_I$}, $\displaystyle\sum_{q'\geq0}\Delta_{q'}\overline{u}^i_I \cdot\widetilde{\Delta}_{q'}J^s\partial_i \overline{\omega}$. By the H\"older inequality, one has
\begin{equation*}
\begin{split}
\big\|\Delta_q\overline{u}^i_I  S_{q-1}\partial_i J^s\overline{\omega}\big\|_{L^2}&\leq \|S_{q-1}\nabla J^s\overline{\omega} \|_{L^\infty}\|\Delta_q\overline{u}_I\|_{L^2}
\\&\leq \sum_{k\leq q-2}\|\Delta_{k}\nabla J^s\overline{\omega} \|_{L^\infty}2^{q(-1+2\alpha)}\|\Delta_q\overline{\omega}\|_{L^2}
\\&\leq 2^{q(1+2\alpha)}\|\Delta_q\overline{\omega}\|_{L^2}\|J^s\overline{\omega}\|_{L^2}\sum_{k\leq q-2}2^{2(k-q)},
\end{split}
\end{equation*}
from which it follows that
\begin{equation*}
\big\|T_{J^s\partial_i \overline{\omega}}\overline{u}^i_I\big\|_{L^2} \leq C\|\overline{\omega}\|_{H^s} \|\overline{\omega}\|_{B^{1+2\alpha}_{2,1}}.
\end{equation*}
For the term $\displaystyle\sum_{q'\geq0}\Delta_{q'}\overline{u}^i_I \widetilde{\Delta}_{q'}J^s\partial_i \overline{\omega}$, by the H\"older inequality, we immediately obtain
\begin{equation*}
\begin{split}
\Big\|\sum_{q'\geq0}\Delta_{q'}\overline{u}^i_I \widetilde{\Delta}_{q'}J^s\partial_i \overline{\omega}\Big\|_{L^2}&\leq
\sum_{q'\geq 0}\big\|\Delta_{q'}\overline{u}^i_I \widetilde{\Delta}_{q'}\partial_i  J^s\overline{\omega}\big\|_{L^2}
\\&\leq\sum_{q'\geq 0}\|\Delta_{q'}\overline{u}_I\|_{L^2}\big\|\widetilde{\Delta}_{q'}\nabla  J^s\overline{\omega}\big\|_{L^\infty}
\\&\leq\|J^s\overline{\omega}\|_{L^2}\sum_{q'\geq 0}\|\Delta_{q'}\overline{\omega}\|_{L^2}
2^{q'(1+2\alpha)}\\&\leq\|\overline{\omega}\|_{H^s}\|\overline{\omega}\|_{B^{1+2\alpha}_{2,1}}.
\end{split}
\end{equation*}

\textbf{Bounds for the term $\displaystyle\sum_{q'\geq0}J^s\big(\Delta_{q'}\overline{u}^i_I \widetilde{\Delta}_{q'}\partial_i \overline{\omega}\big)$}. Utilizing again the H\"{o}lder inequality and  Bernstein's inequality gives
\begin{align*}
\Big\|\sum_{q'\geq0}J^s\big(\Delta_{q'}\overline{u}^i_I \widetilde{\Delta}_{q'}\partial_i \overline{\omega}\big)\Big\|_{L^2}&\leq
\sum_{q'\geq0}\Big(\sum_{q\leq q'+N_0}2^{2qs}\big\|\Delta_{q}(\Delta_{q'}\overline{u}^i_I \widetilde{\Delta}_{q'}\partial_i \overline{\omega})\big\|_{L^2}^2\Big)^\frac12
\\&\leq\sum_{q'\geq0}\Big(\sum_{q\leq q'+N_0}2^{2qs}\Big)^\frac12\|\Delta_{q'}\overline{u}_I\|_{L^2} \big\|\widetilde{\Delta}_{q'}\nabla \overline{\omega}\big\|_{L^\infty}
\\&\leq\sum_{q'\geq0}2^{q's}\|\Delta_{q'}\overline{\omega}\|_{L^2}
2^{q'(1+2\alpha)}\|\widetilde{\Delta}_{q'}\overline{\omega}\|_{L^2}
\\&\leq\|\overline{\omega}\|_{H^s}\|\overline{\omega}\|_{H^{1+2\alpha}}.
\end{align*}

\textbf{Bounds  for the last term $[J^s, {\Delta}_{-1}\overline{u}_I^i\partial_i]\widetilde{\Delta}_{-1}\overline{\omega}$}.  Adopting to the  similar method to estimate $[J^s, T_{\overline{u}^i_I} \partial_i]\overline{\omega}$, we get
\begin{align*}
&[J^s, {\Delta}_{-1}\overline{u}_I^i\partial_i]\widetilde{\Delta}_{-1}\overline{\omega}
\\=&\sum_{|q+1|\leq2}[J^s\Delta_q, {\Delta}_{-1}\overline{u}_I^i\partial_i]\widetilde{\Delta}_{-1}\overline{\omega}
\\=&\sum_{|q+1|\leq2}\int_{\R^2} 2^{2q}G_s(2^qy)\left({\Delta}_{-1}\overline{u}^i_I(x-y)-{\Delta}_{-1}\overline{u}^i_I(x)\right)\widetilde{\Delta}_{-1}\partial_i\overline{\omega}(x-y)\,\mathrm{d}y
\\=&\sum_{|q+1|\leq2}\int_{\R^2}\int_{0}^{1} 2^{2q}G_s(2^qy)\big(y\cdot\nabla  {\Delta}_{-1}\overline{u}^i_I(x-\tau y)\big)\widetilde{\Delta}_{-1}\partial_i\overline{\omega}(x-y)\,\mathrm{d}\tau\mathrm{d}y.
\end{align*}
Based on this, the Minkowski inequality and the H\"older inequality allow us to infer that
\begin{equation*}\begin{split}
&\big\|[J^s, {\Delta}_{-1}\overline{u}_I^i\partial_i]\widetilde{\Delta}_{-1}\overline{\omega}\big\|_{L^2}
\\ \leq&\Big(\sum_{|q+1|\leq2}\Big\|\int_{\R^2}\int_{0}^{1} 2^{2q}G_s(2^qy)\big(y\cdot\nabla  {\Delta}_{-1}\overline{u}^i_I(x-\tau y)\big)\widetilde{\Delta}_{-1}\partial_i\overline{\omega}(x-y)\,\mathrm{d}\tau\mathrm{d}y\Big\|_{L^2}^2\Big)^{\frac12}
\\ \leq&C_s\|{\Delta}_{-1}\nabla\overline{u}_I\|_{L^\infty}
\big\|\widetilde{\Delta}_{-1}\nabla\overline{\omega}\big\|_{L^2}
  \leq C_s\big\|\Lambda^{2\alpha}\overline{\omega}\big\|_{L^2}\|\overline{\omega}\|_{L^2}\leq
C_s\|\overline{\omega}\|_{H^{2\alpha}}^2.
\end{split}
\end{equation*}
Combining these estimates above yields \eqref{add-2}. This ends the proof.
\end{proof}

\section{Proof of main Theorems }\label{sec5}
\setcounter{section}{5}\setcounter{equation}{0}

 This section is devoted to showing the main theorems. Let us begin by  proving Theorem~\ref{th1}.
 \subsection{Proof of Theorem \ref{th1}}

First of all, let us denote
$$\overline{\omega}=\omega^{\alpha}-\omega^{\alpha_{0}}\quad\text{and}\quad\overline{u}=u^{\alpha}-u^{\alpha_{0}}.$$
Then,  the couple $(\overline{\omega},\,\overline{u})$ satisfies
\begin{equation}\label{diffrence}
\overline{\omega}_{t}+u^{\alpha_{0}}\cdot\nabla \overline{\omega}+\overline{u}\cdot\nabla \overline{\omega}+\overline{u}\cdot\nabla \omega^{\alpha_{0}} =0.
\end{equation}
Operating $J^s$ on \eqref{diffrence} and taking the scalar product of the resulting
equation with $J^s\overline{\omega}$ in $L^2,$ we get
\begin{equation}\label{11-1}\begin{split}
\frac12\frac{\mathrm{d}}{\mathrm{d}t}\|\overline{\omega}(t)\|_{H^s}^2=&-\int_{\R^2} J^s(u^{\alpha_{0}}\cdot\nabla \overline{\omega})J^s\overline{\omega}\,\mathrm{d}x-\int_{\R^2} J^s(\overline{u}\cdot\nabla \overline{\omega})J^s\overline{\omega}\,\mathrm{d}x\\
&-\int_{\R^2} J^s(\overline{u}\cdot\nabla \omega^{\alpha_{0}})J^s\overline{\omega}\,\mathrm{d}x\\:=&
I_1+I_2+I_3.
\end{split}\end{equation}

We are going to estimate the three terms on the right hand side of \eqref{11-1} one by one. By the divergence-free condition and \eqref{com-2}, we have
\begin{equation}\label{I1E}\begin{split}
I_1=&-\int_{\R^2} (J^s(u^{\alpha_{0}}\cdot\nabla \overline{\omega})-u^{\alpha_{0}}\cdot\nabla J^s\overline{\omega})
J^s\overline{\omega}\,\mathrm{d}x\\ \leq&
\big\|J^s(u^{\alpha_{0}}\cdot\nabla \overline{\omega})-u^{\alpha_{0}}\cdot\nabla J^s\overline{\omega}\big\|_{L^2}
\|J^s\overline{\omega}\|_{L^2}\\ \leq&\Big(\|J^su^{\alpha_{0}}\|_{L^\frac{1}{\alpha_0}}\|\nabla\overline{\omega}\|_{L^\frac{2}{1-2\alpha_0}}+\|\nabla u^{\alpha_{0}}\|_{L^\infty}\|J^{s-1}\nabla\overline{\omega}\|_{L^2}\Big)
\|J^s\overline{\omega}\|_{L^2}.
\end{split}\end{equation}
Note that $0<\alpha_0<\frac12$ and
$$
u^{\alpha_0}=\nabla^{\perp}(-\Delta)^{-1+\alpha_0}\omega^{\alpha_0}.
$$
Using Lemma \ref{Hardy} yields
$$\|J^su^{\alpha_{0}}\|_{L^\frac{1}{\alpha_0}}\leq C\norm{\omega^{\alpha_0}}_{H^s}.$$
Applying Lemma \ref{embedding1} and Lemma \ref{Hardy} gives
$$\|\nabla\overline{\omega}\|_{L^\frac{2}{1-2\alpha_0}}\leq C\|\overline{\omega}\|_{H^s},$$
\begin{equation*}\begin{split}
\|\nabla u^{\alpha_{0}}\|_{L^\infty}\leq \norm{u^{\alpha_{0}}}_{W^{s,\frac{1}{\alpha_0}}}\leq C\norm{\omega^{\alpha_0}}_{H^{s}}.
\end{split}
\end{equation*}
Thus, \begin{equation}\label{wsp-1}\begin{split}
|I_1|\leq C\|\overline{\omega}\|_{H^s}^2\norm{\omega^{\alpha_0}}_{H^{s}}.
\end{split}\end{equation}
By the decomposition \eqref{diff-1-1}, the second term can be written as
\begin{equation*}\begin{split}
 I_2=&-\int_{\R^2}(J^s(\overline{u}\cdot\nabla \overline{\omega})-\overline{u}\cdot\nabla J^s\overline{\omega})
 J^s\overline{\omega}\,\mathrm{d}x\\=&
-\int_{\R^2}(J^s(\overline{u}_I\cdot\nabla \overline{\omega})-\overline{u}_I\cdot\nabla J^s\overline{\omega})
J^s\overline{\omega}\,\mathrm{d}x-\int_{\R^2}\big( J^s(\overline{u}_{II}\cdot\nabla \overline{\omega})-\overline{u}_{II}\cdot\nabla J^s\overline{\omega}\big)
J^s\overline{\omega}\,\mathrm{d}x.
\end{split}
\end{equation*}
Using  \eqref{com-2} and the Sobolev embedding inequalities, we obtain
\begin{equation*}\begin{split}
&-\int_{\R^2}\big(J^s(\overline{u}_I\cdot\nabla \overline{\omega})-\overline{u}_I\cdot\nabla J^s\overline{\omega}\big)J^s\overline{\omega}\,\mathrm{d}x\\ \leq&\big(\|J^s\overline{u}_I\|_{L^\frac1\alpha}\|\nabla\overline{\omega}\|_{L^\frac{2}{1-2\alpha}}+\|\nabla \overline{u}_I\|_{L^\infty}\|J^{s-1}\nabla\overline{\omega}\|_{L^2}\big)
\|J^s\overline{\omega}\|_{L^2}\\ \leq&\|\overline{\omega}\|_{H^s}^2\|J^s\overline{u}_I\|_{L^\frac1\alpha}.
\end{split}
\end{equation*}
Similarly,  for $p>\frac1\alpha>2,$
\begin{equation*}\begin{split}
&-\int_{\R^2}\big(J^s(\overline{u}_{II}\cdot\nabla \overline{\omega})-\overline{u}_{II}\cdot\nabla J^s\overline{\omega}\big)J^s\overline{\omega}\,\mathrm{d}x\\ \leq&\big(\|J^s\overline{u}_{II}\|_{L^p}\|\nabla\overline{\omega}\|_{L^\frac{2p}{p-2}}+
\|\nabla \overline{u}_{II}\|_{L^\infty}\|J^{s-1}\nabla\overline{\omega}\|_{L^2}\big)
\|J^s\overline{\omega}\|_{L^2}\\ \leq&\|\overline{\omega}\|_{H^s}^2\|J^s\overline{u}_{II}\|_{L^p}.
\end{split}
\end{equation*}
Therefore,
\begin{equation}\label{wsp-3}
|I_2|\leq
\|\overline{\omega}\|_{H^s}^2(\|J^s\overline{u}_I\|_{L^\frac1\alpha}+\|J^s\overline{u}_{II}\|_{L^p}).
\end{equation}
Concerning the third term, we use  \eqref{com-1} and the Sobolev embedding inequalities to get
\begin{equation}\label{wsp-4}\begin{split}
|I_3|=&\Big|\int J^s(\overline{u}\cdot\nabla \omega^{\alpha_{0}})J^s\overline{\omega}\,\mathrm{d}x\Big|\\
=&
\Big|\int J^s(\overline{u}_{I}\cdot\nabla \omega^{\alpha_{0}})J^s\overline{\omega}\,\mathrm{d}x+\int J^s(\overline{u}_{II}\cdot\nabla \omega^{\alpha_{0}})J^s\overline{\omega}\,\mathrm{d}x\Big|
\\ \leq&
\|\overline{\omega}\|_{H^s}(\|J^s\overline{u}_I\|_{L^\frac1\alpha}+\|J^s\overline{u}_{II}\|_{L^p})\|\omega^{\alpha_{0}}\|_{H^{s+1}}.
\end{split}\end{equation}
Plugging  these estimates \eqref{wsp-1},    \eqref{wsp-3},  \eqref{wsp-4} into \eqref{11-1} yields
\begin{equation}\label{Hs-est}\begin{split}
 \frac{\mathrm{d}}{\mathrm{d}t}\|\overline{\omega}(t)\|_{H^s}  \leq&
\|\overline{\omega}\|_{H^s}\big(\|\omega^{\alpha_{0}}\|_{H^{s}}
+\|J^s\overline{u}_I\|_{L^\frac1\alpha}+\|J^s\overline{u}_{II}\|_{L^p}\big)\\
&+
\|\omega^{\alpha_{0}}\|_{H^{s+1}}\big(\|J^s\overline{u}_I\|_{L^\frac1\alpha}+
\|J^s\overline{u}_{II}\|_{L^p}\big)\\
:=&\tilde {I_1}+\tilde {I_2}.
\end{split}
\end{equation}
The integral form of $\overline{u}_I$ can be written as
\begin{equation*}
J^s\overline{u}_I(x)=\int_{\R^2}\frac{(x-y)^{\perp}}{|x-y|^{2+2\alpha}}
J^s\overline{\omega}(y)\,\mathrm{d}y.
\end{equation*}
Then, using  Lemma \ref{Hardy} enables us to get
\begin{equation}\label{4-16-1}
\begin{split}
\|J^s\overline{u}_I\|_{L^{\frac1\alpha}}&\leq\Big\|\int_{\R^2}\frac{1}{|x-y|^{2-(1-2\alpha)}}
|J^s\overline{\omega}(y)|\,\mathrm{d}y\Big\|_{L^{\frac1\alpha}}\\&\leq
C(\alpha)\|J^s\overline{\omega}\|_{L^{2}},
\end{split}\end{equation}
where $C(\alpha)$ depends on $\alpha$ and will be bounded if $0\le \alpha,\alpha_0<\frac12$ (but will be unbounded if $\alpha$ tend to $\frac12$).
When $p>\frac1\alpha>2,$ adopting to the similar way to \eqref{4-16-1} gives
\begin{equation}\label{4-16-2}\begin{split}
\|J^s\overline{u}_{II}\|_{L^p}&\leq C(\alpha)\big(\|J^s\omega^{\alpha_0}\|_{L^{\frac{2p}{2+p(1-2\alpha)}}}+
 \|J^s\omega^{\alpha_0}\|_{L^{\frac{2p}{2+p(1-2\alpha_0)}}}\big)
 \\&\leq C(\alpha)\|\omega^{\alpha_0}\|_{H^{s+1}}.
\end{split}\end{equation}
Hence,
\begin{equation}\label{Hs-1}
\tilde {I_1} \leq  C\|\overline{\omega}\|_{H^s}^{2}+C\|\overline{\omega}\|_{H^{s}}
\|\omega^{\alpha_0}\|_{H^{s+1}}.
\end{equation}
On the other hand, the estimate \eqref{4-16-2} is not adaptable to $\tilde {I_2}$ in \eqref{Hs-est}. We will use a different way to estimate
$\|J^s\overline{u}_{II}\|_{L^p}$ in \eqref{Hs-est}.
For $0<\epsilon<1$ to be determined later, we write  $\overline{u}_{II}$ as
\begin{equation*}\begin{split}
 \overline{u}_{II}=\int_{\R^2}\Big(\frac{(x-y)^{\perp}}{|x-y|^{2+2\alpha}}
-\frac{(x-y)^{\perp}}{|x-y|^{2+2\alpha_0}}\Big)\omega^{\alpha_0}(y)\,\mathrm{d}y.
\end{split}\end{equation*}
Therefore,
\begin{equation}\label{11-6-0}\begin{split}
 J^s\overline{u}_{II}&=\int_{\R^2}\Big(\frac{(x-y)^{\perp}}{|x-y|^{2+2\alpha}}
-\frac{(x-y)^{\perp}}{|x-y|^{2+2\alpha_0}}\Big)J^s\omega^{\alpha_0}(y)\,\mathrm{d}y
\\&=\Big(\int_{|x-y|\leq\epsilon}+\int_{1>|x-y|\geq\epsilon}+\int_{|x-y|\geq1}\Big)\Big(\frac{(x-y)^{\perp}}{|x-y|^{2+2\alpha}}
-\frac{(x-y)^{\perp}}{|x-y|^{2+2\alpha_0}}\Big)J^s\omega^{\alpha_0}(y)\,\mathrm{d}y
\\&:=H_{1}+H_{2}+H_{3}.
\end{split}\end{equation}
For the first term $H_1,$ using the Young inequality and the Sobolev embedding, we get
\begin{equation*}\begin{split}
\|H_1\|_{L^{p}}&= \Big\|\int_{|x-y|\leq\epsilon}\Big(\frac{(x-y)^{\perp}}{|x-y|^{2+2\alpha}}
-\frac{(x-y)^{\perp}}{|x-y|^{2+2\alpha_0}}\Big)J^s\omega^{\alpha_0}(y)\,\mathrm{d}y\Big\|_{L^{p}}
\\&\leq C\Big(\frac{\epsilon^{1-2\alpha}}{1-2\alpha}+\frac{\epsilon^{1-2\alpha_0}}{1-2\alpha_0}\Big)
\|J^s\omega^{\alpha_0}\|_{L^{p}}
\\&\leq C\Big(\frac{\epsilon^{1-2\alpha}}{1-2\alpha}
+\frac{\epsilon^{1-2\alpha_0}}{1-2\alpha_0}\Big)\|\omega^{\alpha_0}\|_{H^{s+1}}.
\end{split}
\end{equation*}
As for $H_2$ and $H_3$, it is divided into two cases.

{\it Case 1: $\alpha_0>\alpha.$} By the mean value theorem, we can obtain
\begin{equation*}\begin{split}
H_2\leq|\alpha_0-\alpha|\int_{1>|x-y|\geq\epsilon}\frac{|\log|x-y||}{|x-y|^{1+2\alpha_0}}
|J^s\omega^{\alpha_0}(y)|\,\mathrm{d}y.
\end{split}
\end{equation*}
Utilizing the Young inequality yields
\begin{equation*}\begin{split}
\|H_2\|_{L^p}&\leq|\alpha_0-\alpha|\|J^s\omega^{\alpha_0}\|_{L^p}
\int_{1>|z|\geq\epsilon}\frac{|\log|z||}{|z|^{1+2\alpha_0}}\,\mathrm{d}z,
\\&\leq\frac{C}{1-2\alpha_0}|\alpha_0-\alpha||\log\epsilon|\|\omega^{\alpha_0}\|_{H^{s+1}}.
\end{split}
\end{equation*}
  To deal with $H_3$, we fix a small number $\sigma>0$ such that $p>\frac{2}{2\alpha-\sigma}>2.$ Thanks to the fact that $\log |x-y|\le C|x-y|^\sigma$ for any $\sigma>0$ and $|x-y|\ge 1$, we have
\begin{equation*}\begin{split}
\|H_3\|_{L^p}&\leq |\alpha_0-\alpha|\Big\|\int_{|x-y|\geq1}\frac{|\log|x-y||}{|x-y|^{1+2\alpha}}
|J^s\omega^{\alpha_0}(y)|\,\mathrm{d}y\Big\|_{L^p}\\&\leq|\alpha_0-\alpha|\|J^s\omega^{\alpha_0}\|_{L^r}
\Big(\int_{|z|\geq1}\Big(\frac{|\log|z||}{|z|^{1+2\alpha}}\Big)^q\,\mathrm{d}z\Big)^{\frac1q}
\\
&\leq|\alpha_0-\alpha|\|J^s\omega^{\alpha_0}\|_{L^r}
\Big(\int_{|z|\geq1}\Big(\frac{1}{|z|^{1+2\alpha-\sigma}}\Big)^q\,\mathrm{d}z\Big)^{\frac1q}\\
&\leq|\alpha_0-\alpha|\|J^s\omega^{\alpha_0}\|_{L^r}\Big(\frac{1}{(1+2\alpha-\sigma)q-2}\Big)^{\frac1q},
\end{split}
\end{equation*}where $\frac1p+1=\frac1r+\frac1q,$ $q>\frac{2}{1+2\alpha-\sigma}$, $p>\frac{2}{2\alpha-\sigma}$, then we can choose some $r>2$ such that
\begin{equation*}
H^{s+1}(\R^2)\hookrightarrow W^{s,r}(\R^2)
\end{equation*}holds (see Lemma \ref{embedding1}).

{\it Case 2: $\alpha_0<\alpha<\frac12.$} It is similar to {\it Case 1} by exchanging the position of $\alpha$ and $\alpha_0$. For instance, by the mean value theorem, we can obtain
\begin{equation*}\begin{split}
H_2\leq|\alpha_0-\alpha|\int_{1>|x-y|\geq\epsilon}\frac{|\log|x-y||}{|x-y|^{1+2\alpha}}
|J^s\omega^{\alpha_0}(y)|\,\mathrm{d}y.
\end{split}
\end{equation*}
Then
\begin{equation*}\begin{split}
\|H_2\|_{L^p}&\leq|\alpha_0-\alpha|\|J^s\omega^{\alpha_0}\|_{L^p}
\int_{1>|z|\geq\epsilon}\frac{|\log|z||}{|z|^{1+2\alpha}}\,\mathrm{d}z,
\\&\leq\frac{C}{1-2\alpha}|\alpha_0-\alpha||\log\epsilon|\|\omega^{\alpha_0}\|_{H^{s+1}}.
\end{split}
\end{equation*}
Now  we fix a small number $\sigma>0$ such that $p>\frac{2}{2\alpha_0-\sigma}>2.$ Similarly, we have
\begin{equation}\label{H3E}\begin{split}
\|H_3\|_{L^p}
&\leq |\alpha_0-\alpha|\Big\|\int_{|x-y|\geq1}\frac{|\log|x-y||}{|x-y|^{1+2\alpha_0}}
|J^s\omega^{\alpha_0}(y)|\,\mathrm{d}y\Big\|_{L^p}
\\
&\leq|\alpha_0-\alpha|\|J^s\omega^{\alpha_0}\|_{L^r}\Big(\frac{1}{(1+2\alpha_0-\sigma)q-2}\Big)^{\frac1q},
\end{split}
\end{equation}where $\frac1p+1=\frac1r+\frac1q,$ $q>\frac{2}{1+2\alpha-\sigma}$, $p>\frac{2}{2\alpha-\sigma}$, then we can choose some $r>2$ such that
\begin{equation*}
H^{s+1}(\R^2)\hookrightarrow W^{s,r}(\R^2)
\end{equation*}holds (see Lemma \ref{embedding1}).

As a consequence, we get
\begin{equation*}
\|J^s\overline{u}_{II}(t)\|_{L^p}\leq C\left(\epsilon^{1-2\alpha}
+\epsilon^{1-2\alpha_0}+|\alpha_0-\alpha|+|\alpha_0-\alpha||\log\epsilon|\right)\|\omega^{\alpha_0}\|_{H^{s+1}}.
\end{equation*}
Hence,
\begin{equation}\label{Hs-2}
\begin{split}
\tilde {I_2}=&\|\omega^{\alpha_{0}}\|_{H^{s+1}}(\|J^s\overline{u}_I\|_{L^\frac1\alpha}+
\|J^s\overline{u}_{II}\|_{L^p})\\ \leq & C\|\omega^{\alpha_{0}}\|_{H^{s+1}}\|\overline{\omega}\|_{H^s}
+\|\omega^{\alpha_{0}}\|_{H^{s+1}}^2\big(\epsilon^{1-2\alpha}+\epsilon^{1-2\alpha_0}
+|\alpha_0-\alpha|+|\alpha_0-\alpha||\log\epsilon|\big).
\end{split}\end{equation}
Set  $\epsilon=|\alpha_0-\alpha|$.
By plugging \eqref{Hs-1} and \eqref{Hs-2} into \eqref{Hs-est}, we get
\begin{equation}\label{Hs-3}\begin{split}
 \frac{\mathrm{d}}{\mathrm{d}t}\|\overline{\omega}(t)\|_{H^s}  \leq&
C\|\omega^{\alpha_0}\|_{H^{s+1}}\|\overline{\omega}\|_{H^s}
 +C\|\overline{\omega}\|_{H^s}^{2}\\
&+\|\omega^{\alpha_{0}}\|_{H^{s+1}}^2\left(|\alpha_0-\alpha|^{1-2\alpha}+|\alpha_0-\alpha|^{1-2\alpha_0}+|\alpha_0-\alpha||\log |\alpha_0-\alpha||\right).
\end{split}\end{equation}
Multiply \eqref{Hs-3} by $\exp(-C\int_0^t\|\omega^{\alpha_0}\|_{H^{s+1}}\,\mathrm{d}s)$ and consider the quantity
$$y(t)=\|\overline{\omega}\|_{H^s}\exp\Big(-C\int_0^t\|\omega^{\alpha_0}\|_{H^{s+1}}\,\mathrm{d}s\Big).$$
We then get the inequality
\begin{equation*}
\frac{\mathrm{d}y(t)}{\mathrm{d}t}\leq \left(|\alpha_0-\alpha|^{1-2\alpha}+|\alpha_0-\alpha|^{1-2\alpha_0}+|\alpha_0-\alpha||\log |\alpha_0-\alpha|\right)F(t)+Gy^2(t),
\end{equation*}
where $$F(t)=\|\omega^{\alpha_{0}}\|_{H^{s+1}}^2\exp\Big(-C\int_0^t\|\omega^{\alpha_0}(s)\|_{H^{s+1}}\,\mathrm{d}s\Big),$$
and $$G=C\exp\Big(C\int_0^T\|\omega^{\alpha_0}(t)\|_{H^{s+1}}\,\mathrm{d}t\Big).$$
By Proposition  \ref{o}, there exists a $\delta>0$ depending on $T$ and $\int_0^T \|\omega^{\alpha_0}\|_{H^{s+1}} \,\mathrm{d}t$ such that when $0<\alpha<\frac12$ and $|\alpha-\alpha_0|<\delta$,
\begin{equation*}
y(t)\leq C\left(|\alpha_0-\alpha|^{1-2\alpha}+|\alpha_0-\alpha|^{1-2\alpha_0}+|\alpha_0-\alpha||\log |\alpha_0-\alpha|\right)\int_0^T F(t)\,\mathrm{d}t,
\end{equation*}
which implies that
\begin{equation*}
\|\overline{\omega}\|_{H^s}\leq C\left(|\alpha_0-\alpha|^{1-2\alpha}+|\alpha_0-\alpha|^{1-2\alpha_0}+|\alpha_0-\alpha||\log|\alpha_0-\alpha||\right).
\end{equation*}
Here $C>0$ is a constant depending on $T$ and $\int_0^T \|\omega^{\alpha_0}\|_{H^{s+1}} \,\mathrm{d}t$ as well.

Assume that  $\omega^{\alpha_{0}}\in C([0,T_0];H^{s+1})$, $s>2.$ According to the local well-posedness theory,  for $\alpha>0$ and $|\alpha-\alpha_0|<\delta$, $\omega^{\alpha}\in C([0,T_{max});H^{s+1})$, where $T_{max}>0$ denotes the maximal existence time.
If $T_{max}\geq T_0$, the proof is finished. If $T_{max}<T_0$, we are going  to prove that $T_{max}$ can be extended to $T_0$.  By performing  the $(s+1)$-order energy estimate, we get
\begin{equation*}\begin{split}
\frac12\frac{\mathrm{d}}{\mathrm{d}t}\|\omega^\alpha(t)\|_{H^{s+1}}^2
&=-\int_{\R^2}\Lambda^{s+1}\big(u^{\alpha}\cdot\nabla \omega^\alpha\big)\Lambda^{s+1}\omega^\alpha\,\mathrm{d}x
\\&=-\int_{\R^2}\Lambda^{s+1}\big(u^{\alpha}\cdot\nabla \omega^\alpha-u^{\alpha}\cdot\Lambda^{s+1}\nabla \omega^\alpha\big)\Lambda^{s+1}\omega^\alpha\,\mathrm{d}x
\\&\leq\|\omega^\alpha\|_{H^{s+1}}\left(\|\Lambda^{s+1}u^\alpha\|_{L^{\frac1\alpha}}\|\nabla \omega^\alpha\|_{L^{\frac{2}{1-2\alpha}}}
+\|\nabla u^\alpha\|_{L^\infty}\|\omega^\alpha\|_{H^{s+1}}\right)
\\&\leq\|\omega^\alpha\|_{H^{s+1}}^2\left(\|\nabla \omega^\alpha\|_{L^{\frac{2}{1-2\alpha}}}+\|\nabla u^\alpha\|_{L^\infty}\right)
\\&\leq\|\omega^\alpha\|_{H^{s+1}}^2\|\omega^\alpha\|_{H^{s}}.
\end{split}\end{equation*}
By the Gronwall inequality, we have
\begin{equation*}\begin{split}
\|\omega^\alpha(t)\|_{H^{s+1}}&\leq e^{\int_0^{T_{max}}\|\omega^\alpha(t)\|_{H^{s}}\,\mathrm{d}t}\|\omega^\alpha_0\|_{H^{s+1}}
\\&\leq e^{\int_0^{T_{max}}(\|\overline{\omega}(t)\|_{H^{s}}+\|\omega^{\alpha_0}(t)\|_{H^{s}})\,\mathrm{d}t}\|\omega^\alpha_0\|_{H^{s+1}}
\leq C
\end{split}\end{equation*}
for $t\in [0,{T_{max}}]$ and hence $\omega^\alpha(T_{max})$ is finite. This deduces a contradiction with $T_{max}$ is the maximal existence time by using the local well-posedness theory. In consequence, $T_{max}=T_0$ as required and the proof of the theorem is finished.

\subsection{Proof of Theorem \ref{th1+}}

Now we prove Theorem \ref{th1+} which corresponds to the case $\alpha_0=0$ (The 2D incompressible Euler equations).  After letting $\alpha_0=0$, we can go through the proof of  Theorem \ref{th1} except the estimates on  $I_1$ in \eqref{11-1},  on the term including $\alpha_0$ in~\eqref{4-16-2} and on the term $H_3$ in \eqref{H3E}.
 For conciseness, it is only needed to give estimates on these terms as follows.

 Firstly, to deal with the term $I_1$ in \eqref{11-1}, we rewrite \eqref{I1E} as
\begin{equation}\label{I1E2}\begin{split}
I_1=&-\int_{\R^2} (J^s(u^{{0}}\cdot\nabla \overline{\omega})-u^{{0}}\cdot\nabla J^s\overline{\omega})
J^s\overline{\omega}\,\mathrm{d}x\\ \leq&
\big\|J^s(u^{{0}}\cdot\nabla \overline{\omega})-u^{{0}}\cdot\nabla J^s\overline{\omega}\big\|_{L^2}
\|J^s\overline{\omega}\|_{L^2}\\ \leq&\Big(\|J^su^{0}\|_{L^p}\|\nabla\overline{\omega}\|_{L^\frac{2p}{p-2}}+\|\nabla u^{{0}}\|_{L^\infty}\|J^{s-1}\nabla\overline{\omega}\|_{L^2}\Big)
\|J^s\overline{\omega}\|_{L^2}
\end{split}\end{equation}
for some $p>2$.

Note that
$$
u^{0}=\nabla^{\perp}(-\Delta)^{-1}\omega^{0}.
$$
Using Lemma \ref{Hardy} yields
\begin{equation}\label{18-9-1}\begin{split}
\|J^su^{{0}}\|_{L^p}\leq C\norm{J^s\omega^{0}}_{L^{\frac{2p}{p+2}}}=C(\norm{\Lambda^s\omega^{0}}_{L^{\frac{2p}{p+2}}}+\norm{\omega^{0}}_{L^{\frac{2p}{p+2}}})
\end{split}
\end{equation}
for some constant $C>0$.

By the Gagliardo-Nirenberg-Sobolev inequality, one has
\begin{equation}\label{18-9-2}\begin{split}
\norm{\omega^{0}}_{L^{\frac{2p}{p+2}}}\le \norm{\omega^{0}}^\theta_{L^2}\norm{\omega^{0}}^{1-\theta}_{L^a},
\end{split}
\end{equation}
where $1\le a<2$ is given in Theorem \ref{th1+} and
$$
\theta=\frac{2-a-\frac{2}{p}a}{2-a}.
$$
To guarantee $0<\theta<1$, it suffices to choose $p>\frac{2a}{2-a}\ge 2$. Moreover, given $p$ satisfying $p>\frac{2a}{2-a}$ and $p\ge\frac{2a(s+1)}{2-a}$, the following interpolation is direct
\begin{equation}\label{18-9-3}\begin{split}
\norm{\Lambda^s\omega^{0}}_{L^{\frac{2p}{p+2}}}\le \norm{\Lambda^{s+1}\omega^{0}}^\theta_{L^2}\norm{\omega^{0}}^{1-\theta}_{L^a}, \quad s>2,
\end{split}
\end{equation}
where $\theta=\frac{s+\frac{2}{a}-1-\frac2p}{s+\frac{2}{a}}$
satisfies $\frac{s}{s+1}\le\theta<1$.

 And
applying Lemma \ref{embedding1} and Lemma \ref{Hardy} gives
\begin{equation}\label{18-9-4}\begin{split}
\|\nabla\overline{\omega}\|_{L^\frac{2p}{p-2}}\leq C\|\overline{\omega}\|_{H^s},
\end{split}
\end{equation}
\begin{equation}\label{18-9-5}\begin{split}
\|\nabla u^{{0}}\|_{L^\infty}\leq \norm{\nabla u^{{0}}}_{H^{s-1}}\leq C\norm{\omega^{_0}}_{H^{s-1}}, \quad s>2.
\end{split}
\end{equation}
Thus, substituting \eqref{18-9-1}-\eqref{18-9-5} into \eqref{I1E2}, we obtain
\begin{equation}\label{18-9-6}\begin{split}
|I_1|\leq C\|\overline{\omega}\|_{H^s}^2(\norm{\omega^{0}}_{H^{s+1}}+\norm{\omega^{0}}_{L^a}).
\end{split}\end{equation}

Next, concerning the term including $\alpha_0$ in \eqref{4-16-2}, similar estimates as in \eqref{18-9-2} and  \eqref{18-9-3} yield
$$
\norm{\omega^{0}}_{L^{\frac{2p}{p+2}}}\le C (\norm{\omega^{0}}_{H^{s+1}}+\norm{\omega^{0}}_{L^a})
$$
and hence the estimate \eqref{4-16-2} can be rewritten as
\begin{equation}\label{18-9-7}\begin{split}
\|J^s\overline{u}_{II}\|_{L^p}\leq C(\|\omega^{0}\|_{H^{s+1}}+\norm{\omega^{0}}_{L^a}),
\end{split}\end{equation}
which implies that \eqref{Hs-1} becomes
\begin{equation}\label{18-9-8}
\tilde {I_1} \leq  C\|\overline{\omega}\|_{H^s}^{2}+C\|\overline{\omega}\|_{H^{s}}
(\|\omega^{0}\|_{H^{s+1}}+\norm{\omega^{0}}_{L^a}).
\end{equation}
Lastly, concerning the term $H_3$ in \eqref{H3E},   we fix a small number $\sigma>0$ such that
\begin{equation}\label{18-9-9}\begin{split}
\|H_3\|_{L^p}
&\leq \alpha\Big\|\int_{|x-y|\geq1}\frac{|\log|x-y||}{|x-y|}
|J^s\omega^{0}(y)|\,\mathrm{d}y\Big\|_{L^p}
\\
&\leq\alpha\|J^s\omega^{0}\|_{L^r}\Big(\frac{1}{(1-\sigma)q-2}\Big)^{\frac1q},
\end{split}
\end{equation}
where $\frac1p+1=\frac1r+\frac1q,$ $q>\frac{2}{1-\sigma}$ and $r<2$.

Similar to  \eqref{18-9-2} and \eqref{18-9-3}, we have
\begin{equation}\label{18-9-10}\begin{split}
\norm{\omega^{0}}_{L^r}\le \norm{\omega^{0}}^\theta_{L^2}\norm{\omega^{0}}^{1-\theta}_{L^a},
\end{split}
\end{equation}
where $1\le a<2$ is given in Theorem \ref{th1+} and
$$
\theta=\frac{\frac1r-\frac12}{\frac1a-\frac12}.
$$
To guarantee $0<\theta<1$, it suffices to choose $a<r<2$, which implies that $p>\frac{2a}{2-a(\sigma+1)}$.  Moreover, given $r$ satisfying $a<r<2$ and $\frac{2a(s+1)}{as+2}\le r<2$ , it holds
\begin{equation}\label{18-9-11}\begin{split}
\norm{\Lambda^s\omega^{0}}_{L^r}\le \norm{\Lambda^{s+1}\omega^{0}}^\theta_{L^2}\norm{\omega^{0}}^{1-\theta}_{L^a}, \quad s>2,
\end{split}
\end{equation}
where $\theta=\frac{s+\frac2a-\frac2r}{s+\frac2a}$ lies in $[\frac{s}{s+1},1)$.

It follows from \eqref{18-9-10} and \eqref{18-9-11} that
\begin{equation}\label{18-9-12}\begin{split}
\|J^s\omega^{0}\|_{L^r}\le C(\|\omega^{0}\|_{H^{s+1}}+\norm{\omega^{0}}_{L^a}).
\end{split}
\end{equation}
Putting \eqref{18-9-12} into \eqref{18-9-9} yields
\begin{equation}\label{18-9-13}\begin{split}
\|H_3\|_{L^p}\le C(\|\omega^{0}\|_{H^{s+1}}+\norm{\omega^{0}}_{L^a})
\end{split}
\end{equation}
with $1\le a<2$.

 Applying estimates \eqref{18-9-6}, \eqref{18-9-8} and \eqref{18-9-12} obtained above and other parts same as in the proof of Theorem \ref{th1} with $\alpha_0=0$,  we finally can rewrite \eqref{Hs-3} as
\begin{equation}\label{Hs-31}\begin{split}
 \frac{\mathrm{d}}{\mathrm{d}t}\|\overline{\omega}(t)\|_{H^s}  \leq&
C(\|\omega^{0}\|_{H^{s+1}}+\|\omega^0\|_{L^a})\|\overline{\omega}\|_{H^s}
 +C\|\overline{\omega}\|_{H^s}^{2}\\
&+\|\omega^{{0}}\|_{H^{s+1}}(\|\omega^{0}\|_{H^{s+1}}+\|\omega^0\|_{L^a})\left(\alpha^{1-2\alpha}+\alpha+\alpha|\log |\alpha||\right),
\end{split}\end{equation}
which yields
\begin{equation*}
\|\overline{\omega}\|_{H^s}\leq C\left(\alpha^{1-2\alpha}+\alpha(1+|\log|\alpha||)\right).
\end{equation*}
Here $C>0$ is a constant depending on $T$ and $\int_0^T \|\omega^{\alpha_0}\|_{H^{s+1}\cap L^a} \,\mathrm{d}t$ as well for any $1\le a<2$. The proof of Theorem \ref{th1+} can be finished as in Theorem \ref{th1}.

\subsection{Proof of Theorem \ref{th2}}

Taking the scalar product of \eqref{diffrence} with $\overline{\omega}$ in $H^s$ and using Lemma \ref{commutator} and the Sobolev embedding inequalities enable us to get
\begin{equation}\label{Hs-est-1-1}\begin{split}
\frac{\mathrm{d}}{\mathrm{d}t}\|\overline{\omega}(t)\|_{H^s}&\leq
\|\overline{\omega}\|_{H^s}\|u^{\alpha_{0}}\|_{H^{s}}+
(\|\overline{\omega}\|_{H^s}+\|\omega^{\alpha_{0}}\|_{H^{s+1}})\|\overline{u}\|_{H^{s}}\\
&\leq
\|\overline{\omega}\|_{H^s}\|\omega^{\alpha_{0}}\|_{H^{s}}+
(\|\overline{\omega}\|_{H^s}+\|\omega^{\alpha_{0}}\|_{H^{s+1}})(\|\overline{u}_{I}\|_{H^{s}}+\|\overline{u}_{II}\|_{H^{s}}).
\end{split}\end{equation}
Here we have used the decomposition \eqref{diff-1-1} with
\begin{equation*}
\overline{u}_{I}=\nabla^{\perp}(-\Delta)^{-1+\alpha}\overline{\omega}, \quad \overline{u}_{II}=\left(\nabla^{\perp}(-\Delta)^{-1+\alpha}-\nabla^{\perp}(-\Delta)^{-1+\alpha_0}\right)\omega^{\alpha_0}.
\end{equation*}
By using Proposition \ref{uni-est} (Remark \ref{Rm2-1})  with $\beta=1-2\alpha$, we obtain
\begin{equation}\label{u1-est}
\|\overline{u}_{I}(t)\|_{H^{s}}\leq
C\big(\|\overline{\omega}\|_{H^s}+(1-2\alpha)\|\overline{\omega}\|_{L^1}\big),
\end{equation}where $C=C(\alpha,s)$ is an absolutely constant when $\alpha\rightarrow\frac12.$
By using Proposition \ref{uni-est} again ($\beta=1-2\alpha$ and $\beta=0$ respectively), there also exists a uniformly bounded constant $C=C(\alpha,s)$ when $\alpha\rightarrow\frac12$ such that
 \begin{equation*}\| \overline{u}_{II}(t)\|_{H^s}\leq C\big(\|\omega^{\alpha_0}\|_{H^s}+\|\omega^{\alpha_0}\|_{L^1}\big).
\end{equation*}
It follows that \begin{equation}\label{equ-12}
\begin{split}
\|\overline{\omega}\|_{H^s}(\|\overline{u}_{I}\|_{H^{s}}+\|\overline{u}_{II}\|_{H^{s}})\leq C\|\overline{\omega}\|_{H^s}^2+
C\|\overline{\omega}\|_{H^s}(\|\omega^{\alpha_0}\|_{H^s}+
\|\omega^{\alpha_0}\|_{L^1}+\|\omega^{\alpha}\|_{L^1}).
\end{split}\end{equation}

Now we adopt to anther way to estimate $\|\overline{u}_{II}\|_{H^s}$ in order to deal with $\|\omega^{\alpha_0}\|_{H^{s+1}}\|\overline{u}_{II}\|_{H^s}$ on the right hand side of \eqref{Hs-est-1-1}. The decomposition \eqref{11-6-0} will be applied, which is
\begin{equation}\label{11-6}\begin{split}
 J^s\overline{u}_{II}&=\int_{\R^2}\Big(\frac{(x-y)^{\perp}}{|x-y|^{2+2\alpha}}
-\frac{(x-y)^{\perp}}{|x-y|^{2+2\alpha_0}}\Big)J^s\omega^{\alpha_0}(y)\,\mathrm{d}y
\\&=\Big(\int_{|x-y|\leq\epsilon}+\int_{1>|x-y|\geq\epsilon}+\int_{|x-y|\geq1}\Big)\Big(\frac{(x-y)^{\perp}}{|x-y|^{2+2\alpha}}
-\frac{(x-y)^{\perp}}{|x-y|^{2+2\alpha_0}}\Big)J^s\omega^{\alpha_0}(y)\,\mathrm{d}y
\\&=H_{1}+H_{2}+H_{3},
\end{split}\end{equation}
where  $0<\epsilon<1$ is to be determined later.

Performing the fact that $$\int_{|x|=1}\frac{x^\perp}{|x|^{2+2\alpha}}\,\mathrm{d}s=\int_{|x|=1}\frac{x^\perp}{|x|^{3}}\,\mathrm{d}s=0,$$
we get
\begin{equation*}\begin{split}
H_{1}&=\int_{|x-y|\leq\epsilon}\Big(\frac{(x-y)^{\perp}}{|x-y|^{2+2\alpha}}-\frac{(x-y)^{\perp}}{|x-y|^{3}}\Big)
\Big(J^s\omega^{\alpha_0}(y)-J^s\omega^{\alpha_0}(x)\Big)\,\mathrm{d}y
\\&=\int_{|z|\leq\epsilon}\Big(\frac{z^{\perp}}{|z|^{2+2\alpha}}-\frac{z^{\perp}}{|z|^{3}}\Big)\big(J^s\omega^{\alpha_0}(x-z)-J^s\omega^{\alpha_0}(x)\big)\,\mathrm{d}z.
\end{split}
\end{equation*}
From the mean value theorem, we deduce that
\begin{equation*}\begin{split}
|H_{1}|&\leq\int_0^1\int_{|z|\leq\epsilon}
\Big(\frac{1}{|z|^{2\alpha}}+\frac{1}{|z|}\Big)\big|\nabla J^s\omega^{\alpha_0}(x-\tau z)\big|\,\mathrm{d}z\mathrm{d}\tau.
\end{split}
\end{equation*}
Now, taking the $L^2$ norm of the above inequality, and using the fact that the norm of an integral is less that the integral of the norm, we get
\begin{equation*}\begin{split}
\|H_{1}\|_{L^2}&\leq \int_0^1\int_{|z|\leq\epsilon}
\Big(\frac{1}{|z|^{2\alpha}}+\frac{1}{|z|}\Big)\big\|\nabla J^s\omega^{\alpha_0}(\cdot-\tau z)\big\|_{L^2}\,\mathrm{d}z\mathrm{d}\tau.
\end{split}
\end{equation*}
The translation invariance of the Lebesgue measure then ensures that
\begin{equation}\label{11-7}\begin{split}
\|H_{1}\|_{L^2}&\leq C\Big(\frac{1}{2-2\alpha}\epsilon^{2-2\alpha}+\epsilon\Big)\big\|J^{s+1}\omega^{\alpha_0}\big\|_{L^2}.
\end{split}
\end{equation}
For $2+2\alpha\leq\xi\leq3,$ we estimate $H_{2}$ as follows,
\begin{equation}\label{11-8}\begin{split}
\|H_{2}\|_{L^2}&= \Big(\frac12-\alpha\Big)\Big\|\int_{1>|x-y|\geq\epsilon}\frac{(x-y)^{\perp}\big(|x-y|^{\xi}\log|x-y|\big)}{|x-y|^{2+2\alpha}|x-y|^{2+2\alpha_0}}
J^s\omega^{\alpha_0}(y)\,\mathrm{d}y\Big\|_{L^2}
\\&\leq \Big(\frac12-\alpha\Big)\Big\|\int_{1>|x-y|\geq\epsilon}\frac{|\log|x-y||}{|x-y|^{2}}
|J^s\omega^{\alpha_0}(y)|\,\mathrm{d}y\Big\|_{L^2}\\
&\leq\Big(\frac12-\alpha\Big)|\log\epsilon|^2\|J^s\omega^{\alpha_0}\|_{L^2},
\end{split}
\end{equation}
where the Young inequality and the mean value theorem have been used.
Adopting to the similar method to estimate $H_{2}$, we get
\begin{equation}\label{11-9}\begin{split}
\|H_{3}\|_{L^2}&\leq \Big(\frac12-\alpha\Big)\Big\|\int_{|x-y|\geq1}\frac{|\log|x-y||}{|x-y|^{1+2\alpha}}
|J^s\omega^{\alpha_0}(y)|\,\mathrm{d}y\Big\|_{L^2}\\&\leq(\frac12-\alpha)\|J^s\omega^{\alpha_0}\|_{L^p}
\Big(\int_{|z|\geq1}\Big(\frac{|\log|z||}{|z|^{1+2\alpha}}\Big)^q\,\mathrm{d}z\Big)^{\frac1q}
\\&\leq\Big(\frac12-\alpha\Big)\|J^s\omega^{\alpha_0}\|_{L^p}\Big(\frac{1}{(1+2\alpha-\sigma)q-2}\Big)^{\frac1q},
\end{split}
\end{equation} where $\frac32=\frac1p+\frac1q,$ $q>\frac{2}{1+2\alpha-\sigma}$ whence $p<\frac{2}{2-2\alpha+\sigma}<2$.
By the Gagliardo-Nirenberg inequality, we have
\begin{equation*}\begin{split}
\|\Lambda^s\omega^{\alpha_0}\|_{L^p}\leq\|\omega^{\alpha_0}\|_{L^1}^{1-\theta} \|\Lambda^{s+1}\omega^{\alpha_0}\|_{L^2}^{\theta},
\end{split}
\end{equation*}
where $\theta=1-\frac{2}{p(s+2)}.$

 Then, we conclude that $p\geq \frac{2(s+1)}{s+2}.$
This enables us to choose some $\frac{2(s+1)}{s+2}\leq p<\frac{2}{2-2\alpha+\sigma}.$
Combining the estimates \eqref{11-7}-\eqref{11-9} with \eqref{11-6}
and choosing $\epsilon=(\frac12-\alpha)$, we get
\begin{equation}\label{u2-est}\begin{split}
\|\overline{u}_{II}(t)\|_{H^{s}}\leq CL(\alpha)\big(\|\omega^{\alpha_0}\|_{H^{s+1}}+\|\omega^{\alpha_0}\|_{L^1}\big).
\end{split}
\end{equation}
Here and what in follow,
\[L(\alpha):=\Big(\frac12-\alpha\Big)^{2-2\alpha}
+\Big(\frac12-\alpha\Big)\Big|\log\Big(\frac12-\alpha\Big)\Big|^2
+\Big(\frac12-\alpha\Big). \]
Hence, combining \eqref{u1-est} and \eqref{u2-est} yields
\begin{equation}\label{equ-12-1}\begin{split}
&\|\omega^{\alpha_{0}}\|_{H^{s+1}}\|\overline{u}_{I}\|_{H^{s}}
+\|\omega^{\alpha_{0}}\|_{H^{s+1}}\|\overline{u}_{II}\|_{H^{s}}
\\ \leq&\|\omega^{\alpha_{0}}\|_{H^{s+1}}\|\overline{\omega}\|_{H^s}+
\Big(\frac12-\alpha\Big)\|\overline{\omega}\|_{L^1}\|\omega^{\alpha_{0}}\|_{H^{s+1}}
 +CL(\alpha)\big(\|\omega^{\alpha_0}\|_{H^{s+1}}^2+\|\omega^{\alpha_0}\|_{L^1}^2\big)
\\ \leq&\|\omega^{\alpha_{0}}\|_{H^{s+1}}
\|\overline{\omega}\|_{H^s}+CL(\alpha)\big(\|\omega^{\alpha_0}\|_{H^{s+1}}^2+
\|\omega^{\alpha_0}\|_{L^1}^2+\|\omega^{\alpha}\|_{L^1}^2\big).
\end{split}
\end{equation}
Plugging \eqref{equ-12} and  \eqref{equ-12-1} into \eqref{Hs-est-1-1} gives
\begin{equation}\label{Hs-est-s}\begin{split}
 \frac{\mathrm{d}}{\mathrm{d}t}\|\overline{\omega}(t)\|_{H^s}\leq&
\|\overline{\omega}\|_{H^s}(\|\omega^{\alpha_{0}}\|_{H^{s+1}}
+\|\omega^{\alpha_0}\|_{L^1}+\|\omega^{\alpha}\|_{L^1})+\|\overline{\omega}\|_{H^s}^2
\\&+CL(\alpha)(\|\omega^{\alpha_0}\|_{H^{s+1}}^2+
\|\omega^{\alpha_0}\|_{L^1}^2+\|\omega^{\alpha}\|_{L^1}^2).
\end{split}
\end{equation}
Thanks to the incompressible condition that $\nabla\cdot u=0$, it follows that  $\|\omega^{\alpha_0}\|_{L^1}+\|\omega^{\alpha}\|_{L^1}$ is bounded if  the initial data $\omega_0\in L^1$. Arguing similarly as the last part in the proof of Theorem \ref{th1}, we obtain
\begin{equation*}
\|\overline{\omega}(t)\|_{H^s}\leq C\left(\Big(\frac12-\alpha\Big)+\Big(\frac12-\alpha\Big)\log^2\Big(\frac12-\alpha\Big)\right).
\end{equation*}
Moreover, we can prove that $\omega^{\alpha}\in C([0,T];H^{s+1})$. The proof of the theorem is finished.

\subsection{Proof of Theorem \ref{th3}}

Similar to the proof Theorem \ref{th1}, it follows from the difference equation \eqref{diffrence} that \eqref{11-1} holds true.
The three terms $I_1,\, I_2,\, I_3$ on the right side of \eqref{11-1} will be estimated as follows. Applying the commutator estimates in Lemma \ref{commutator} and the Sobolev embedding inequalities, we immediately have
\begin{equation}\label{11-2}\begin{split}
|I_1|=&\Big|\int_{\R^2} \big(J^s(u^{\alpha_{0}}\cdot\nabla \overline{\omega})-u^{\alpha_{0}}\cdot\nabla J^s\overline{\omega}\big)
J^s\overline{\omega}\,\mathrm{d}x\Big|\\ \leq&
\|J^s(u^{\alpha_{0}}\cdot\nabla \overline{\omega})-u^{\alpha_{0}}\cdot\nabla J^s\overline{\omega}\|_{L^2}
\|J^s\overline{\omega}\|_{L^2}\\ \leq&(\|J^su^{\alpha_{0}}\|_{L^2}\|\nabla\overline{\omega}\|_{L^\infty}+\|\nabla u^{\alpha_{0}}\|_{L^\infty}\|J^{s-1}\nabla\overline{\omega}\|_{L^2})
\|J^s\overline{\omega}\|_{L^2}\\ \leq&\|\overline{\omega}\|_{H^s}^2\|u^{\alpha_{0}}\|_{H^{s}}.
\end{split}\end{equation}
Substituting the decomposition \eqref{diff-1-1} into $I_2$ and $I_3$, respectively, we have
\begin{equation}\label{I_2}\begin{split}
&I_2=-\int_{\R^2} J^s(\overline{u}_{I}\cdot\nabla \overline{\omega})J^s\overline{\omega}\,\mathrm{d}x-\int_{\R^2} J^s(\overline{u}_{II}\cdot\nabla \overline{\omega})J^s\overline{\omega}\,\mathrm{d}x:= I_{21}+I_{22};\\
&I_3=-\int_{\R^2} J^s(\overline{u}_{I}\cdot\nabla \omega^{\alpha_{0}})J^s\overline{\omega}\,\mathrm{d}x-\int_{\R^2} J^s(\overline{u}_{II}\cdot\nabla \omega^{\alpha_{0}})J^s\overline{\omega}\,\mathrm{d}x:= I_{31}+I_{32}.
\end{split}\end{equation}
Choose some $\sigma\in (0,2\alpha).$  For any $p>\frac{2}{2\alpha-\sigma}>2,$ applying Lemma \ref{commutator} and the Sobolev embedding inequalities again, we obtain
\begin{equation*}\begin{split}
|I_{22}|=&\Big|\int_{\R^2}\big(J^s(\overline{u}_{II}\cdot\nabla \overline{\omega})-\overline{u}_{II}\cdot\nabla J^s\overline{\omega}\big)J^s\overline{\omega}\,\mathrm{d}x\Big|\\ \leq&\|J^s(\overline{u}_{II}\cdot \nabla \overline{\omega})-\overline{u}_{II}\cdot J^s \nabla \overline{\omega}\|_{L^2}\|\overline{\omega}\|_{H^s}
\\ \leq &\big(\|J^s\overline{u}_{II}\|_{L^p}\|\nabla\overline{\omega}\|_{L^\frac{2p}{p-2}}+
\|\nabla \overline{u}_{II}\|_{L^\infty}\|J^{s-1}\nabla\overline{\omega}\|_{L^2}\big)
\|J^s\overline{\omega}\|_{L^2}\\ \leq&\|\overline{\omega}\|_{H^s}^2\|J^s\overline{u}_{II}\|_{L^p},
\end{split}
\end{equation*}
and
\begin{equation*}\begin{split}
|I_{32}|&\leq
\|J^s(\overline{u}_{II}\cdot\nabla \omega^{\alpha_{0}})\|_{L^2}\|\overline{\omega}\|_{H^s}
\\&\leq\|\overline{\omega}\|_{H^s}\|\omega^{\alpha_0}\|_{H^{s+1}}\|J^s\overline{u}_{II}\|_{L^p}.
\end{split}\end{equation*}
Applying Proposition \ref{add1} gives
\begin{equation*}\begin{split}
|I_{21}|=&\Big|\int_{\R^2}\big(J^s(\overline{u}_I\cdot\nabla \overline{\omega})-\overline{u}_I\cdot\nabla J^s\overline{\omega}\big)J^s\overline{\omega}\,\mathrm{d}x\Big|\\ \leq&\|J^s(\overline{u}_I\cdot \nabla \overline{\omega})-\overline{u}_I\cdot J^s \nabla \overline{\omega}\|_{L^2}\|\overline{\omega}\|_{H^s}
 \leq \|\overline{\omega}\|^3_{H^s}.
\end{split}
\end{equation*}
By Proposition \ref{add-0}, one has
\begin{equation*}\begin{split}
|I_{31}|&\leq
\|J^s(\overline{u}_I\cdot\nabla \omega^{\alpha_{0}})\|_{L^2}\|\overline{\omega}\|_{H^s}
\\&\leq\|\overline{\omega}\|_{H^s}^2\|\omega^{\alpha_0}\|_{H^{s+2}}.
\end{split}\end{equation*}
Inserting the estimates of $I_{21},\,I_{22}, \,I_{31}$ and $I_{32}$ into \eqref{I_2}, we arrive at
\begin{equation}\label{11-3}\begin{split}
&|I_2|+|I_3|\\
\leq& C\left(\|\overline{\omega}\|_{H^s}^2\|J^s\overline{u}_{II}\|_{L^p}+
\|\overline{\omega}\|_{H^s}\|\omega^{\alpha_0}\|_{H^{s+1}}\|J^s\overline{u}_{II}\|_{L^p}+
\|\overline{\omega}\|^3_{H^s}+\|\overline{\omega}\|_{H^s}^2\|\omega^{\alpha_0}\|_{H^{s+2}}\right).
\end{split}\end{equation}
In view of \eqref{11-1}, \eqref{11-2} and \eqref{11-3}, it deduces
\begin{equation}\label{th3-proof}\begin{split}
 \frac{\mathrm{d}}{\mathrm{d}t}\|\overline{\omega}(t)\|_{H^s}  \leq&\|\overline{\omega}\|_{H^s}(\|u^{\alpha_{0}}\|_{H^{s}}+\|\omega^{\alpha_0}\|_{H^{s+2}})
+\|\overline{\omega}\|^2_{H^s}\\
&+\|\overline{\omega}\|_{H^s}\|J^s\overline{u}_{II}\|_{L^p}
+\|\omega^{\alpha_0}\|_{H^{s+1}}\|J^s\overline{u}_{II}\|_{L^p}.
\end{split}\end{equation}

Now we estimate $\|J^s\overline{u}_{II}\|_{L^p}$. Similar to the proof of Theorem \ref{th2}, we use the integral form \eqref{11-6}.
Note that
\begin{equation*}\begin{split}
H_{1}&=\int_{|x-y|\leq\epsilon}\Big(\frac{(x-y)^{\perp}}{|x-y|^{2+2\alpha}}-\frac{(x-y)^{\perp}}{|x-y|^{3}}\Big)\Big(J^s\omega^{\alpha_0}(y)-J^s\omega^{\alpha_0}(x)\Big)\,\mathrm{d}y.
\end{split}
\end{equation*}
By the mean value formula and the H\"{o}lder inequality, we obtain
\begin{equation*}\begin{split}
\|H_{1}\|_{L^p}&\leq C\Big(\frac{1}{2-2\alpha}\epsilon^{2-2\alpha}+\epsilon\Big)\|J^{s+1}\omega^{\alpha_0}\|_{L^p}
\\&\leq C\Big(\frac{1}{2-2\alpha}\epsilon^{2-2\alpha}+\epsilon\Big)\|\omega^{\alpha_0}\|_{H^{s+2}},
\end{split}
\end{equation*}
where Lemma \ref{embedding1} is used in the last inequality.
Similarly, for $2+2\alpha\leq\gamma\leq3,$
\begin{equation*}\begin{split}
\|H_{2}\|_{L^p}&= \Big(\frac12-\alpha\Big)\Big\|\int_{1>|x-y|\geq\epsilon}\frac{(x-y)^{\perp}(|x-y|^{\gamma}\log|x-y|)}{|x-y|^{2+2\alpha}|x-y|^{2+2\alpha_0}}
J^s\omega^{\alpha_0}(y)\,\mathrm{d}y\Big\|_{L^p}
\\&\leq \Big(\frac12-\alpha\Big)\Big\|\int_{1>|x-y|\geq\epsilon}\frac{|\log|x-y||}{|x-y|^{2}}
|J^s\omega^{\alpha_0}(y)|\,\mathrm{d}y\Big\|_{L^p}\\&\leq\Big(\frac12-\alpha\Big)|\log\epsilon|^2\|J^s\omega^{\alpha_0}\|_{L^p}
\\&\leq\Big(\frac12-\alpha\Big)|\log\epsilon|^2\|\omega^{\alpha_0}\|_{H^{s+1}}.
\end{split}
\end{equation*}
Note that $\ln |x-y|\le C|x-y|^\sigma$ for $|x-y|\ge 1$ and any $\sigma>0$, where $C$ may depend on $\sigma$. We can apply the mean value formula and the Young inequality to obtain
\begin{equation*}\begin{split}
\|H_{3}\|_{L^p}&\leq\Big (\frac12-\alpha\Big)\Big\|\int_{|x-y|\geq1}\frac{|\log|x-y||}{|x-y|^{1+2\alpha}}
|J^s\omega^{\alpha_0}(y)|\,\mathrm{d}y\Big\|_{L^p}\\&\leq\Big(\frac12-\alpha\Big)\|J^s\omega^{\alpha_0}\|_{L^r}
\Big(\int_{|z|\geq1}\Big(\frac{|\log|z||}{|z|^{1+2\alpha}}\Big)^q\,\mathrm{d}z\Big)^{\frac1q}
\\&\leq\Big(\frac12-\alpha\Big)\|J^s\omega^{\alpha_0}\|_{L^r}\Big(\frac{1}{(1+2\alpha-\sigma)q-2}\Big)^{\frac1q},
\end{split}
\end{equation*}
where $\frac1p+1=\frac1r+\frac1q,$ $q>\frac{2}{1+2\alpha-\sigma}$, $p>\frac{2}{2\alpha-\sigma}$, and we can choose some $r>2$ such that the following embedding
\begin{equation*}
H^{s+1}(\R^2)\hookrightarrow W^{s,r}(\R^2)
\end{equation*}
holds (see Lemma \ref{embedding1}).
Consequently,
\begin{equation}\label{11-5}\begin{split}
 \|J^s\overline{u}_{II}(t)\|_{L^p}\leq &\|H_{1}\|_{L^p}+\|H_{2}\|_{L^p}+\|H_{3}\|_{L^p}
 \\\leq & C\Big(\frac{1}{2-2\alpha}\epsilon^{2-2\alpha}+\epsilon\Big)\|\omega^{\alpha_0}\|_{H^{s+2}}
\\
&+\Big(\frac12-\alpha\Big)|\log\epsilon|^2\|\omega^{\alpha_0}\|_{H^{s+1}}+\Big(\frac12-\alpha\Big)\|\omega^{\alpha_0}\|_{H^{s+1}}.
\end{split}
\end{equation}
Let  $\epsilon=\frac12-\alpha$. The estimate \eqref{11-5} combined with \eqref{th3-proof} yields
\begin{equation*}\begin{split}
\frac{\mathrm{d}}{\mathrm{d}t}\|\overline{\omega}(t)\|_{H^s}  \leq &\|\overline{\omega}\|_{H^s}\big(\|u^{\alpha_{0}}\|_{H^{s}}
+\|\omega^{\alpha_0}\|_{H^{s+2}}\big)\\
&+\|\overline{\omega}\|^2_{H^s}+\|\omega^{\alpha_0}\|_{H^{s+2}}^2\left(\Big(\frac12-\alpha\Big)+\Big(\frac12-\alpha\Big)\log^2\Big(\frac12-\alpha\Big)\right).
\end{split}\end{equation*}
Arguing similarly as the last part in the proof of Theorem \ref{th1}, we obtain
\begin{equation*}
\|\overline{\omega}(t)\|_{H^s}\leq C\left(\Big(\frac12-\alpha\Big)+\Big(\frac12-\alpha)\log^2\Big(\frac12-\alpha\Big)\right).
\end{equation*}
Moreover, we can prove that $\omega^{\alpha}\in C([0,T]; H^{s+2})$ for any $t\in [0,T]$. The proof of the theorem is finished.

\subsection{Proof of Corollary \ref{Cor1+}}

Suppose that the result is not true. Then there exists a $M>0$ and a $\delta_0>$ such that $T^*_\alpha\le M$ for all $\alpha\in (0,\delta_0)$. But it is known that for any $T>0$, the smooth solution of the Euler equations exists on $[0,T]$. Take $T=M+1$. According to Theorem \ref{th1+}, there exists a $0<\delta\le \delta_0$ depending on $T$ such that the smooth solution of the generalized SQG exists on $[0,T]$ as well. This contradicts with the assumption that  the maximal existence time $T^*_\alpha\le M$. The proof of the corollary is complete.

\appendix

\section{ Littlewood-Paley theory and Besov spaces}\label{sec4}
\setcounter{section}{6}\setcounter{equation}{0}

 In this appendix, we introduce Besov spaces which are a generalization of Sobolev spaces. We  recall the dyadic decomposition of the unity in the whole space (see e.g. \cite{[1],[MWZ]}).
\begin{proposition}\label{Dy}
There exists a couple of smooth functions $(\chi,\varphi)$ with  values in
$[0,1]$  such that
$\mathrm{supp}\,\chi\subset\big\{\xi\in\mathbb{R}^{n}:|\xi|\leq\frac{4}{3}\big\}$,
$\mathrm{supp}\,\varphi\subset\big\{\xi\in\mathbb{R}^{n}:\frac{3}{4}\leq|\xi|\leq\frac{8}{3}\big\}$
and
\begin{enumerate}[\rm (i)]
    \item $\chi(\xi)+\displaystyle{\sum_{j\in \mathbb{N}}}\varphi(2^{-j}\xi)=1,$ $\forall\,\xi\in \mathbb{R}^{n},$\\[2mm]
\item $\mathrm{supp}\, \varphi(2^{-p}\cdot)\cap \mathrm{supp}\, \varphi(2^{-q}\cdot)=\emptyset,$ if\, $|p-q|\geq2,$\\[2mm]
\item $\mathrm{supp}\, \chi(\cdot)\cap \mathrm{supp}\, \varphi(2^{-q}\cdot)=\emptyset,$ if\, $q\geq1.$
\end{enumerate}

\end{proposition}
For every $u\in \mathcal{S}'(\mathbb{R}^{n})$, we define the nonhomogeneous Littlewood-Paley  operators by
\begin{equation*}
 \Delta_{-1}u:=\chi(D)u\quad\text{and}\quad   {\Delta}_{j}u:=\varphi(2^{-j}D)u\quad {\rm for\ each\ }j\in\mathbb{N}.
\end{equation*}
We shall also use the following low-frequency cut-off:
\begin{equation*}
    {S}_{j}u:=\chi(2^{-j}D)u.
\end{equation*}
It may be  easily checked that
\begin{equation*}
    u=\sum_{j\geq-1}{\Delta}_{j}u
\end{equation*}
holds in $\mathcal{S}'(\mathbb{R}^{n})$.

The following lemma is the well-known Bernstein inequality which has been frequently used in the proof of Proposition \ref{add-0} and Proposition \ref{add1}.

\begin{lemma}[\textbf{Bernstein's inequality}]\label{B}
Let $\mathcal {B}$ be a ball of  $\R^{n}$, and $\mathcal {C}$ be a
ring of $\R^{n}$. There exists a positive constant C such that for
all integer $k\geq0$, all $1\leq a\leq b \leq\infty$ and
$u\in{L^{a}(\R^{n})}$, the following estimates are satisfied:
$$\sup_{|\alpha|=k} \|\partial^{\alpha}u\|_{L^{b}(\R^{n})}\leq
C^{k+1}\lambda^{k+n(\frac{1}{a}-\frac{1}{b})}\|u\|_{L^{a}(\R^{n})},
~~ \mathrm{supp}\, \hat{u}\subset \lambda \mathcal {B},$$
$$C^{-(k+1)}\lambda^{k}\|u\|_{L^{a}(\R^{n})}\leq \sup_{|\alpha|=k}\|\partial^{\alpha}u\|_{L^{a}(\R^{n})}\leq C^{k+1}\lambda^{k}\|u\|_{L^{a}(\R^{n})},
~~ \mathrm{supp}\, \hat{u}\subset \lambda \mathcal {C}.$$
\end{lemma}

Let us now introduce the basic tool of the paradifferential calculus which is Bony's decomposition. That is,
for two tempered distributions $u$ and $v$,
\begin{equation*}
uv=T_{u}v+T_{v}u+R(u,v),
\end{equation*}where
$$
T_{u}v=\displaystyle{\sum_{j}}S_{j-1}u\Delta_{j}v,  \qquad R(u,v)=\displaystyle{\sum_{j}}\Delta_{j}u\widetilde{\Delta}_{j}v,
$$
where $\widetilde{\Delta}_{j}=\Delta_{j-1}+\Delta_{j}+\Delta_{j+1}.$ In usual,
$T_{u}v$ is called paraproduct of $v$ by $u$ and $R(u,v)$ the remainder term.

\begin{definition}\label{Besov space}
For $s\in \mathbb{R}$, $(p,q)\in [1,+\infty]^{2}$ and $u\in \mathcal{S}'(\mathbb{R}^{n})$, we set
\begin{equation*}
    \norm{u}_{{B}^{s}_{p,q}(\mathbb{R}^{n})}:=
    \Big(\sum_{j\geq-1}2^{jsq}
    \norm{{\Delta}_{j}u}_{L^{p}(\mathbb{R}^{n})}^{q}\Big)^{\frac{1}{q}}
    \quad\text{if}\quad q<+\infty
\end{equation*}
and
\begin{equation*}
    \norm{u}_{{B}^{s}_{p,\infty}(\mathbb{R}^{n})}:=\sup_{j\geq-1}2^{js}\norm{{\Delta}_{j}u}_{L^{p}(\mathbb{R}^{n})}.
\end{equation*}
Then we define  inhomogeneous Besov spaces as
\begin{equation*}
   {B}^{s}_{p,q}(\mathbb{R}^{n}):=\big\{u\in\mathcal{S}'(\R^{n}):\,\,\norm{u}_{{B}^{s}_{p,q}(\mathbb{R}^{n})}<+\infty\big\}.
\end{equation*}
\end{definition}
It should be remarked that the usual Sobolev spaces $H^s(\R^2)$  coincide with the inhomogeneous Besov spaces  $B_{2,2}^s(\R^n)$.
Also, by using Definition \ref{Besov space},  we get easily  for any $s>2$ and $\alpha<\frac12,$ the following embeds hold
\begin{equation}\label{imbedding}\tag{D.1}
H^s(\R^n)\hookrightarrow B^{1+2\alpha}_{2,1}(\R^n)\hookrightarrow H^{1+2\alpha}(\R^n).
\end{equation}
Lastly, we turn to review two useful lemmas which have been used in foregoing sections.
\begin{lemma}\label{annulus}
Let $\mathcal {C'}$ be an  annulus of $\R^n$, $s$ be a real number, and $[p,r]\in[1,\infty]^2.$ Assume $\{u_j\}_{j\geq-1}$ be a sequence of smooth functions such that $$\mathrm{supp}\, \widehat{u}_j\subset2^j\mathcal {C'} \quad\text{and}\quad\big\|\{2^{js}\|u_j\|_{L^p(\R^n)}\}_{j\geq-1}\big\|_{\ell^r}<\infty.$$
We then have $$u:=\sum_{j\geq-1} u_j\in {B}^{s}_{p,r}(\R^n)\quad\text{and}\quad\|u\|_{{B}^{s}_{p,r}(\R^n)}\leq C_s\big\|\{2^{js}\|u_j\|_{L^p(\R^n)}\}_{j\geq-1}\big\|_{\ell^r}.$$
\end{lemma}
\begin{lemma}\label{ball}
Let $\mathcal {B'}$ be a ball of $\R^n$, $s>0$ be a real number, and $[p,r]\in[1,\infty]^2.$ Let $\{u_j\}_{j\geq-1}$ be a sequence of smooth functions such that $$\mathrm{supp}\, \widehat{u}_j\subset2^j\mathcal {B'}\quad\text{and}\quad\big\|\{2^{js}\|u_j\|_{L^p(\R^n)}\}_{j\geq-1}\big\|_{\ell^r}<\infty.$$
We then have $$u:=\sum_{j\geq-1} u_j\in {B}^{s}_{p,r}(\R^n)\quad\text{and}\quad\|u\|_{{B}^{s}_{p,r}(\R^n)}\leq C_s\big\|\{2^{js}\|u_j\|_{L^p(\R^n)}\}_{j\geq-1}\big\|_{\ell^r}.$$
\end{lemma}

{\bf Acknowledgements.}
 Jiu was partially supported  by the National Natural Science Foundation of China
(No.11671273).  Zheng was partially supported  by the National Natural Science Foundation of China (No.11501020, No.11771423).

\vspace{3mm}

{\it Conflict of Interest:} The authors declare that they have no conflict of interest.

\end{document}